\documentclass[a4paper,11pt]{article}
\usepackage{amsmath,amssymb,amsthm,multirow}
\usepackage{graphicx,epsfig,color}
\usepackage{fancyhdr,subcaption,graphicx}
\usepackage{amsfonts}
\usepackage{mathtools} 
\usepackage{latexsym}
\usepackage{algorithm,algorithmic}
\usepackage{pstricks}
\usepackage{bbm}
\usepackage{todonotes,booktabs}
\usepackage{acronym,comment}

\usepackage[a4paper, margin=2cm]{geometry}

\definecolor{lightgreen}{rgb}{0.22,0.50,0.25}
\definecolor{lightblue}{rgb}{0.22,0.45,0.70}
\usepackage[colorlinks=true,breaklinks=true,linkcolor=lightgreen,citecolor=lightblue]{hyperref}

\newtheorem{theorem}{Theorem}[section]
\newtheorem{lemma}[theorem]{Lemma}

\numberwithin{equation}{section}
\numberwithin{figure}{section}
\numberwithin{table}{section}

\newcommand\bbR{\mathbb{R}}

\newcommand\bbQ{\mathbb{Q}}

\newcommand\bH{{\mathbf H}}
\newcommand\bV{{\mathbf V}}

\newcommand{\bn}{\boldsymbol{n}}
\newcommand{\bu}{\boldsymbol{u}}
\newcommand{\bv}{\boldsymbol{v}}
\newcommand{\bw}{\boldsymbol{w}}
\newcommand{\bx}{\boldsymbol{x}}
\newcommand{\bz}{\boldsymbol{\zeta}}
\newcommand{\ten}{\mathsf}

\newcommand{\bnabla}{\boldsymbol{\nabla}}
\newcommand{\bsigma}{\boldsymbol{\sigma}}
\newcommand{\btau}{\boldsymbol{\tau}}

\newcommand\dx{{\,\mathrm{d}\bx\,}}
\newcommand\ds{{\,\mathrm{d}s\,}}
\newcommand\dt{{\Delta t}}

\def\tr{\mathrm{tr}}

\acrodef{fe}[FE]{Finite Element}
\acrodef{fem}[FEM]{Finite Element Method}
\acrodef{pde}[PDE]{Partial Differential Equation}
\acrodefplural{pde}[PDEs]{Partial Differential Equations}
\acrodef{amr}[AMR]{Adaptive Mesh Refinement and coarsening}
\acrodef{aa}[AA]{Anderson Acceleration}
\acrodef{dir}[DIR]{Deformable Image Registration problem}
\acrodef{dof}[DoF]{Degree of Freedom}
\acrodefplural{dof}[DoFs]{Degrees of Freedom}

\title{Tree-based adaptive finite element methods for deformable image registration\thanks{\textbf{Funding:} This work has been partially supported by the Monash Mathematics Research Fund S05802-3951284; by the Australian Research Council through the \textsc{Future Fellowship} grant FT220100496 and \textsc{Discovery Project} grant DP22010316; and by the National Research and Development Agency (ANID) of the Ministry of Science, Technology, Knowledge and Innovation of Chile through the grant FONDECYT de Postdoctorado No. 3230326 and the Center for Mathematical Modeling (CMM), Proyecto Basal FB210005. This work was also supported by computational resources provided by the Australian Government through NCI under the ANU Merit Allocation Scheme (ANUMAS).}}
\author{Nicol\'as A. Barnafi\thanks{Instituto de Ingeniería Matemática y Computacional \& Facultad de Ciencias Biológicas, Pontificia Universidad Católica de Chile, Av Vicuña Mackenna 4860, Santiago, Chile, and Center for Mathematical Modeling, Santiago, Chile.  Email: {\tt nicolas.barnafi@uc.cl}.} \and Alberto F. Mart\'in\thanks{School of Computing, Australian National University, Acton ACT 2601, Australia. Email: { \tt alberto.f.martin@anu.edu.au}.} \and Ricardo Ruiz-Baier\thanks{School of Mathematics, Monash University, 9 Rainforest Walk, Melbourne 3800 VIC, Australia; and Universidad Adventista de Chile, Casilla 7-D Chill\'an, Chile. Email: {\tt ricardo.ruizbaier@monash.edu}.}}
\begin{document}
\maketitle

\begin{abstract}
    In this work we propose an adaptive \ac{fem} formulation for the \ac{dir} together with a residual-based \emph{a posteriori} error estimator, whose efficiency and reliability are theoretically established. This estimator is used to guide \ac{amr}. The nonlinear Euler--Lagrange equations associated with the minimisation of the relevant functional are solved with a pseudo time-stepping fixed-point scheme which is  further accelerated using \ac{aa}. The efficient implementation of these solvers relies on an efficient adaptive mesh data structure based on forests-of-octrees endowed with space-filling-curves. Several numerical results illustrate the performance of the proposed methods applied to adaptive \ac{dir} in application-oriented problems.
\end{abstract}

\noindent
{\bf Keywords}: Deformable image registration, a-posteriori error estimates, Anderson acceleration, adaptive mesh refinement and coarsening, forest-of-octrees.

\smallskip\noindent
{\bf Mathematics subject classification (2020)}: 65N50, 65N15, 74B20, 74N25.

\maketitle

\section{Introduction}
Image registration problems consist in finding an appropriate mapping of a given template image so that it resembles another reference image. These images can originate from several different applications, and may represent physical parameters, locations, or even time. See \cite{sotiras2013deformable} for a review in medical imaging. There are many ways of defining such a problem, including variational formulations \cite{christensen1996deformable,henn2004multimodal}, level-set methods, grid deformation methods \cite{lee2010optimal}, learning processes \cite{de2019deep} and Bayesian techniques \cite{deshpande2019bayesian}, among others. All these approaches differ in how they measure image alignment, with such measure being referred to as \emph{similarity}. Image registration is named \emph{deformable} whenever the mapping between images is allowed to vary arbitrarily between pixels, which results in a highly nonlinear problem.

Image registration remains a challenging problem in computational imaging due to three inherent difficulties: its ill-posed nature, the selection of appropriate similarity metrics, and the high computational demands of numerical optimisation. First, the problem's ill-posedness necessitates regularisation strategies, such as elastic potential energy constraints \cite{sotiras2013deformable}, to ensure physically plausible solutions. Second, defining effective similarity measures--whether through energy-based criteria or intensity difference metrics like the $L^2$ norm--requires careful consideration of image modality and noise characteristics. Third, solving the Euler--Lagrange equations derived from the energy minimisation framework imposes significant computational costs \cite{haber2007octree}, particularly when capturing localised deformations. These deformations often exhibit high spatial gradients to resolve fine anatomical or structural details, further amplifying discretisation challenges.

Adaptive numerical methods naturally present a promising approach to address these issues. By dynamically refining computational grids or basis functions in regions of sharp deformation, adaptive techniques hold the promise to achieve a remarkable trade-off among accuracy and computational efficiency. Recent advances in \ac{amr} demonstrate the potential of such strategies to overcome traditional limitations in registration tasks, for which we highlight \cite{barnafi2021adaptive,haber2007octree,haber2008adaptive,pawar2016adaptive,zhang2013adaptive}.
With this in mind, in this work we leverage hierarchically-adapted non-conforming forest-of-octrees meshes endowed with Morton (a.k.a., Z-shaped) space-filling-curves for storage and data partitioning; see, e.g., \cite{badia2020generic,burstedde2011p4est}. These $n$-cube meshes (made of quadrilaterals or cubes in 2D and 3D, resp.) can be very efficiently handled (i.e. refined, coarsened, re-partitioned, etc.) using high-performance and low-memory footprint algorithms \cite{burstedde2011p4est}. However, they are non-conforming at the interfaces of cells with different refinement levels as they yield hanging vertices, edges and faces. In order to enforce the conformity (continuity) of \ac{fe} spaces built out of this kind of meshes, it is standard practice to equip the space with additional linear multi-point constraints. The structure and set up of such constraints is well-established knowledge, and thus not covered here; see \cite{badia2020generic} and references therein for further details.

In this work, we consider the \ac{dir} with an $L^2$ similarity measure and a regularisation given by the (linear) elastic energy of the deformation with Robin and pure Neumann boundary conditions. (We note that nonlinear elasticity has also been considered as an alternative regularization term in other works, such as, e.g. \cite{genet2018equilibrated}.) This formulation was analysed in \cite{barnafi2018primal}, and an efficient and reliable \emph{a-posterior} error estimator was developed for it in \cite{barnafi2021adaptive} for the pure-Neumann case, where uniqueness was obtained by imposing orthogonality against the kernel of the regularisation operator. We extend that analysis to the case of Robin boundary conditions. This efficient and reliable \emph{a-posterior} error estimator is used at a given solution of the \ac{dir} on a given mesh in order to automatically adapt it, and the process is repeated across several \ac{amr} iterations to successively improve the accuracy of the \ac{dir} solution.

As mentioned above, solving \ac{dir} is a challenging task. It is essentially a nonlinear inverse problem. A well-established strategy to solve it is by means of a pseudo-time formulation \cite{modersitzki2003numerical}, which is an IMEX approximation of a proximal point algorithm \cite{kaplan1998proximal}. In this work, we leverage \ac{aa} \cite{walker2011anderson} to improve its robustness and efficiency in terms of the iteration count. 
Anderson acceleration is a method that considers a fixed-point iteration and yields another one with improved convergence, both in terms of the radius of convergence and the convergence rate \cite{toth2015convergence}. This technique has already been successfully validated for proximal-point algorithms in the context of inverse problems \cite{mai2020anderson}, which is the method that most resembles the scheme we use. 

The main contributions of this work are: (i) the extension of the efficiency and reliability proof of the existing \emph{a posteriori} error indicator for \ac{dir} to Robin boundary conditions; (ii) the use of \ac{aa} to accelerate the pseudo-time solution strategy commonly used in \ac{dir}; (iii) the use of the computed \emph{a posteriori} error estimator to guide efficient \ac{amr} using octree-based meshes; and (iv) the combination of \ac{aa} and \ac{amr} into a novel \ac{aa}-\ac{amr} algorithm for solving \ac{dir}. These methodologies result in significant computational savings and increased accuracy, as shown in our numerical tests. Our software is implemented in the Julia programming language \cite{bezanson2017julia} using the \ac{fe} tools provided by the \texttt{Gridap} ecosystem of Julia packages \cite{verdugo22,badia22,GridapP4est2025}.

\paragraph{Structure.} The remainder of the work has been organised as follows. In Section~\ref{section:problem setting}, we formulate the \ac{dir} problem. In Section~\ref{section:fem}, we propose  \ac{fem} discretisation for \ac{dir} and its related robust (in the sense of reliability and efficiency) \emph{a-posteriori} error estimator driving the \ac{amr} process. In Section~\ref{section:solution strategy}, we show how we orchestrate the solution of the nonlinear problem with both \ac{aa} and \ac{amr}. In Section~\ref{section:numerical tests} we provide several numerical tests to validate our approach. We conclude our work with some comments and an outlook for future work in Section~\ref{section:conclusions}. 
\section{Problem setting}\label{section:problem setting}
Consider a domain $\Omega\subset \bbR^{d=2,3}$, and two fields $R:\Omega \to \bbR$ and $T:\Omega\to \bbR$ referred to as \emph{reference} and \emph{target} (or template) images, where $R(\bx)$ and $T(\bx)$ denote the \emph{image intensity} at point $\bx$. \ac{dir} consists in finding a transformation under which the template image resembles the reference \cite{modersitzki2003numerical}, that is to find a mapping of $T$ onto $R$ by means of a warping $\bu$ 
such that 
	\begin{equation}\label{intro:equality}
		T(\bx + \bu(\bx)) \approx R(\bx) \quad \forall \, \bx \in \Omega\,.
	\end{equation}
	This can be reformulated as a variational problem 
	\begin{equation}\label{eqn:pvar}
\inf_{\bu\in \mathcal{V}}\alpha \mathcal{D}[\bu;R,T]+\mathcal{S}[\bu],
\end{equation}
where $\mathcal{V}$ models the admissible deformations space, $\mathcal{D}: \mathcal{V}\rightarrow\mathbb{R}$ is the aforementioned \emph{similarity measure}  which attains its minimum when \eqref{intro:equality} holds, $\mathcal S$ is a regulariser, and $\alpha$ is a regularisation parameter that balances the effects of $\mathcal D$ and $\mathcal S$. In this work we consider $\mathcal V=\bH^1(\Omega)$,  the $L^2$-norm of the error for the similarity, given by 
    \begin{equation*}
        \mathcal D[\bu; R, T] = \int_\Omega (T(\bx + \bu(\bx)) - R(\bx))^2\dx,
    \end{equation*}
and as a regulariser we use the elastic deformation energy, defined in the following manner 
    \begin{equation*}
        \mathcal{S}[\bu]:=\frac{1}{2}\int_{\Omega}{\bf \mathcal{C}}\mathbf{e}(\bu):\mathbf{e}(\bu)\dx.
    \end{equation*}
Here 
$\mathbf{e}(\bu)=\frac{1}{2}\{\bnabla \bu+(\bnabla \bu)^{\mathrm{t}}\}$ 
is the infinitesimal strain tensor, i.e., the symmetric component of the displacement field gradient, and ${\bf \mathcal{C}}$ is the elasticity tensor for isotropic solids:
\[
    {\bf \mathcal{C}}\btau=\lambda \tr(\btau)\mathbb{I}+2\mu\btau\quad \forall\btau\in\bbR^{d\times d}.
\]
Naturally, one may consider many other types of regularisations. See \cite{modersitzki2003numerical} for a review.
Assuming that (\ref{eqn:pvar}) has at least one solution with sufficient regularity, the associated Euler--Lagrange equations deliver the following strong problem with Robin or Neumann boundary conditions (representing springs of stiffness $\kappa\geq 0$): Find $\bu$ in $\bH^1(\Omega)$ such that
\begin{equation}\label{eqn:problem}
\begin{array}{rlll}       
 -\mathbf{div}({\bf \mathcal{C}}\mathbf{e}(\bu))& = &\alpha\boldsymbol{f_u}&\quad  \mathrm{in} \quad\Omega,\\[1ex]
{\bf \mathcal{C}}
\mathbf{e}(\bu)\bn + \kappa \bu & = & \boldsymbol{0}& \quad   \mathrm{on} \quad \partial\Omega,
 \end{array}
 \end{equation}
 where 
 \begin{equation}\label{eqn:problem-f}
\boldsymbol{f_u}(\boldsymbol{x})=\big\{T(\boldsymbol{x}+\bu(\boldsymbol{x}))-R(\boldsymbol{x})\big\}\nabla T(\boldsymbol{x}+\bu(\boldsymbol{x}))\quad \forall\boldsymbol{x} \in\Omega.
 \end{equation}
 For the analysis we assume that there are positive constants $L_f$ and $M_f$ such that the nonlinear load term $\boldsymbol{f_u}$ is Lipschitz continuous and uniformly bounded:
     \begin{equation}\label{eqn:assumps}
         |\boldsymbol{f_u}(\boldsymbol{x})-\boldsymbol{f_v}(\boldsymbol{x})| \leq L_f|\boldsymbol{u}(\boldsymbol{x})-\boldsymbol{v}(\boldsymbol{x})|, \quad 
          |\boldsymbol{f_u}(\boldsymbol{x})|\leq M_f\qquad  \forall\,\boldsymbol{x}\in \Omega\,\, a.e.
     \end{equation}

 Under the previous considerations, \eqref{eqn:pvar} becomes 
	\begin{equation*}
        \min_{\bv\in \bH^1(\Omega)}\,\Big\{ \alpha  \int_\Omega |T(\bx + \bu(\bx)) - R(\bx)|^2\dx +\frac12 \int_\Omega \mathcal{C}\mathbf{e}(\bu):\mathbf{e}(\bu)\dx\Big\} + \kappa \int_{\partial\Omega}|\bu|^2\ds,
	\end{equation*}
with first order conditions given by the primal variational formulation for the registration problem: 
Find $\bu\in \bH^1(\Omega)$ such that
 \begin{equation}\label{eqn:prim}
 a(\bu,\bv)=\alpha F_{\bu}(\bv)\quad \forall\bv\in \bH^1(\Omega),
 \end{equation}
 where $a:\bH^1(\Omega)\times\bH^1(\Omega)\rightarrow\mathbb{R}$ is the bilinear form defined by
  \begin{equation}\label{a-primal}
      a(\bu,\bv) \coloneqq\int_{\Omega}{\bf \mathcal{C}}\mathbf{e}(\bu):\mathbf{e}(\bv)+\kappa\int_{\partial\Omega}\bu\cdot \bv\ds \quad\quad \forall\,\bu,\bv\in \bH^1(\Omega),   
  \end{equation}
  and for every $\bu\in\bH^1(\Omega)$, $F_{\bu}:\bH^1(\Omega)\rightarrow\mathbb{R}$ is the linear functional given by
  \begin{equation}
  F_{\bu}(\bv)   \coloneqq-\int_{\Omega}\boldsymbol{f}_{\bu}\cdot \bv \qquad\qquad\forall\, \bv \in\bH^1(\Omega).
  \end{equation}
The conditions (\ref{eqn:assumps}) imply  the Lipschitz continuity and uniform boundedness of $F_{\bu}$, that is
  \begin{equation}\label{eqn:lipschitz}
  \| F_{\bu}- F_{\bv}\|_{{\bH^1(\Omega)}^{\prime}}\leq L_F\|\bu-\bv\|_{0,\Omega}, \qquad 
   \| F_{\bu}\|_{{\bH^1(\Omega)}^{\prime}}\leq M_F 
  \qquad \forall\,\bu,\bv\in \bH^1(\Omega),
  \end{equation}
  respectively. We recall the results concerning the  solvability  of (\ref{eqn:prim}), as developed in \cite[Section 3]{barnafi2018primal}. For it, whenever $\kappa=0$ we need to modify the solution space in order to guarantee uniqueness. We will denote with $\mathbb{RM} = \ker(\mathcal S)$ as the space of rigid body modes, and thus consider the solution space given by 
    \begin{equation*}
        \bV = \begin{cases} \bH^1(\Omega) & \kappa>0 \\ \bH^1(\Omega) \cap \left[\mathbb{RM}\right]^\perp & \kappa = 0\end{cases}.
    \end{equation*}
With it, we define the following linear auxiliary problem: Given $\bz\in \bH^1(\Omega)$, find $\bu\in \bV$ such that
   \begin{equation}\label{eqn:prim-z}
   a(\bu,\bv)=\alpha F_{\bz}(\bv),\quad \bv\in \bV.
   \end{equation}
   
   \begin{theorem}
Given $\bz\in \bH^1(\Omega)$, problem  \eqref{eqn:prim-z} has a unique solution $\bu\in \bV$, and there exists  $C_p > 0$ such that 
$$\|\bu\|_{1,\Omega}\leq \alpha C_p\| F_{\bz}\|_{{\bH^1(\Omega)}^{\prime}}\,.$$
\end{theorem}

We now define the operator $\mathbf{\widehat T} : \bV\rightarrow \bV$ given by $\mathbf{\widehat T}(\bz)=\bu$, where $\bu$ is the unique solution to problem (\ref{eqn:prim-z}) and thus rewrite (\ref{eqn:prim}) as the fixed-point equation: Find $\bu$ in $\bV$ such that
\begin{equation}\label{T}
\mathbf{\widehat T}(\bu)=\bu.
\end{equation}

The following result, {also proven in \cite[Theorem 3]{barnafi2018primal}, establishes the existence and uniqueness} of solution to the fixed-point equation (\ref{T}) whenever $\kappa=0$.
\begin{theorem}
Under data assumptions {\rm(\ref{eqn:assumps})}, the operator $\mathbf{\widehat T}$ has at least one fixed point. Moreover, if $\alpha C_p L_F < 1$, the fixed point is unique.
\end{theorem}
Throughout the manuscript, we will use $\bV$ to write our problem in order to avoid technicalities regarding the implementation of rigid body modes $\mathbb{RM}$ whenever $\kappa=0$. We enforce it in practice by means of an adequate Lagrange multiplier as in \cite{barnafi2018primal}. 

\section{Adaptive discretisation scheme}\label{section:fem}
Let $\bV_h$ be a finite dimensional subspace of $\bV$ built out of a suitable mesh partition $\mathcal T_h$ of $\Omega$ into quadrilateral or hexahedral elements, where $h$ denotes the mesh size. The primal nonlinear discrete problem consists in finding  $\bu_h\in \bV_h$ such that
     \begin{equation}\label{eqn:primdisc}
    a(\bu_h,\bv_h)=\alpha F_{\bu_h}(\bv_h)\quad \forall\bv_h\in \bV_h.
    \end{equation}
Analogously to the continuous case, we consider the auxiliary linear problem: Given $\bz_h\in \bV_h$, find $\bu_h\in \bV_h$ such that
     \begin{equation}\label{eqn:primdisc-zH}
     a(\bu_h,\bv_h)=\alpha F_{\bz_h}(\bv_h)\quad \forall\bv_h\in\bV_h,
     \end{equation}
and also let  $T_h: \bV_h\rightarrow \bV_h$ be the discrete operator given by $T_h(\bz_h)=\bu_h$, where $\bu_h$ is the solution to problem (\ref{eqn:primdisc-zH}). Considering the same data assumptions as in the continuous case, as well as the continuity and bound obtained before, we arrive at the following result proven in \cite[Theorem 5]{barnafi2018primal}, adapted to include the terms arising from $\kappa\neq 0$.
\begin{theorem}
Assume that data assumptions {\rm(\ref{eqn:assumps})} hold. Then, the operator $T_h$ has at least one fixed point. Moreover, if $\alpha C_p L_F < 1$, then such fixed point is unique.
\end{theorem}
Next, we define for each $K\in\mathcal{T}_h$ the \emph{a-posteriori} error indicator
\begin{align}\label{eqn:estimatorprim}
\nonumber   \Theta_K^2 & \coloneqq h_K^2\|\alpha\boldsymbol{f}_{\bu_h}\! - \textbf{div}({\bf \mathcal{C}}\mathbf{e}(\bu_h))\|_{0,K}^2 +\!\!\!\sum_{e\in\mathcal{E}(K)\cap\mathcal{E}_h(\Omega)}\!\!\! h_e\|[{\bf \mathcal{C}}\mathbf{e}(\bu_h)\bn_e]\|_{0,e}^2 \\
&\qquad        + \!\!\! \sum_{e\in\mathcal{E}(K)\cap\mathcal{E}_h(\Gamma)}\!\!\!h_e\|{\bf \mathcal{C}}\mathbf{e}(\bu_h)\bn_e+\kappa\bu_h\|_{0,e}^2,
 \end{align}
where, according to (\ref{eqn:problem-f}),
$$\left.\boldsymbol{f}_{\bu_h}\right|_K(\boldsymbol{x})\coloneqq\big\{T(\boldsymbol{x}+\bu_h(\boldsymbol{x}))-R(\boldsymbol{x})\big\}\nabla T(\boldsymbol{x}+\bu_h(\boldsymbol{x}))\quad \forall\,\boldsymbol{x}\in K,$$
and introduce the global {\it a-posteriori} error estimator
\begin{equation*}
 \Theta\coloneqq\left\{ \sum_{K\in\mathcal{T}_h}\Theta_K^2\right\}^{1/2}.
\end{equation*}
 The following theorem constitutes the main result of this section.
 
\begin{theorem}
Let  $\bu\in \bV$ and  $\bu_h\in \bV_h$ be the solutions of {\rm(\ref{eqn:prim})} and {\rm(\ref{eqn:primdisc})}, respectively, and assume
 that $\alpha C_p L_F<1/2$. Then,  there exist positive constants $h_0,\,C_{\mathrm{rel}},\,C_{\mathrm{eff}}$ independent of $h$ such that for $h\leq h_0 $ there holds
\begin{equation}\label{eqn:cotasprim}
 C_{\mathrm{eff}}\Theta \leq \|\bu-\bu_h\|_{1,\Omega} \leq  C_{\mathrm{rel}}\Theta.
\end{equation}
\end{theorem}

{The reliability (upper bound in (\ref{eqn:cotasprim})) and the efficiency (lower bound in (\ref{eqn:cotasprim})) of $\Theta$  are established separately in the following two lemmas. }
\begin{lemma}\label{lem:relprim}
Assume that $\alpha C_pL_F<1/2$. Then,  there exist positive constants $h_0,\,C_{\mathrm{rel}}$ independent of $h$, such that for $h\leq h_0 $ there holds
 $$\|\bu-\bu_h\|_{\bV} \leq  C_{\mathrm{rel}}\,\Theta.$$
\end{lemma}
 \begin{proof}
Let us first define
$$\mathcal{R}_h(\bw-\bw_h)\coloneqq\alpha F_{\bu}(\bw-\bw_h)-a(\bu_h,\bw-\bw_h)\quad \forall\bw_h\in \bV_h.$$
 As a consequence of the ellipticity of $a$ (c.f. (\ref{a-primal})) with ellipticity constant $\bar{\alpha}$ (c.f. \cite[Corollary 11.2.22]{brenner2008mathematical}), we obtain the following condition
 \begin{equation*}
 \bar{\alpha}\|\bv\|_{1,\Omega} \leq\sup_{\substack{\bw\in \bH^1(\Omega)\\ \bw\neq\boldsymbol{0}}}\frac{a(\bv,\bw)}{\|\bw\|_{\bV}}\quad \forall\bv\in \bH^1(\Omega).
 \end{equation*}
 In particular, for $\bv\coloneqq\bu-\bu_h\in \bH^1(\Omega)$, we notice from (\ref{eqn:prim}) and (\ref{eqn:primdisc}) that $a(\bu-\bu_h,\bw_h)=0$ $\,\forall\bw_h\in \bV_h$, and hence we obtain $a(\bu-\bu_h,\bw)=a(\bu-\bu_h,\bw-\bw_h)=\mathcal{R}_h(\bw-\bw_h)$, which yields
 \begin{equation}\label{eqn:issus}
 \bar{\alpha}\|\bu-\bu_h\|_{1,\Omega} \leq\sup_{\substack{\bw\in \bH^1(\Omega)\\ \bw\neq\boldsymbol{0}}}\frac{\mathcal{R}_h(\bw-\bw_h)}{\|\bw\|_{1,\Omega}}\quad \forall\bw_h\in \bV_h.
 \end{equation}
 From the definition of $\mathcal{R}_h(\bw-\bw_h)$, integrating by parts on each $K\in\mathcal{T}_h$, and adding and subtracting  a suitable term, we can write
 \begin{equation}\label{eqn:R}
 \begin{array}{l}
 \mathcal{R}_h(\bw-\bw_h)=\,\alpha F_{\bu_h}(\bw-\bw_h)+\alpha F_{\bu}(\bw-\bw_h)-a(\bu_h,\bw-\bw_h)-\alpha F_{\bu_h}(\bw-\bw_h)\\
              \quad \displaystyle{=\,\alpha\left\{F_{\bu}(\bw-\bw_h)-F_{\bu_h}(\bw-\bw_h)\right\}-\alpha\int_{\Omega}\boldsymbol{f}_{\bu_h}\cdot (\bw-\bw_h)}
             \\
 \qquad               \displaystyle{\,-\sum_{K\in\mathcal{T}_h}\int_K{\bf \mathcal{C}}\mathbf{e}(\bu_h):\mathbf{e}(\bw-\bw_h)} \\
           \quad   =\displaystyle{\,\alpha\{(F_{\bu}-F_{\bu_h})(\bw-\bw_h)\}-\alpha\int_{\Omega}\boldsymbol{f}_{\bu_h}\cdot (\bw-\bw_h)}\\
              \qquad \displaystyle{\,-\sum_{K\in\mathcal{T}_h} \left\{-\int_K\mathbf{div}({\bf \mathcal{C}}\mathbf{e}(\bu_h))\cdot(\bw-\bw_h)+\int_{\partial K}({\bf \mathcal{C}}\mathbf{e}(\bu_h) \bn_e)\cdot(\bw-\bw_h)\right\},}\\
              \quad \displaystyle{=\,\alpha\{(F_{\bu}-F_{\bu_h})(\bw-\bw_h)\}+\sum_{K\in\mathcal{T}_h}\int_K(\mathbf{div}({\bf \mathcal{C}}\mathbf{e}(\bu_h))-\alpha\boldsymbol{f}_{\bu_h})\cdot (\bw-\bw_h)}\\
               \qquad\displaystyle{\,-\sum_{e\in\mathcal{E}_h(\Omega)}\int_e\left[ ({\bf \mathcal{C}}\mathbf{e}(\bu_h) \bn_e)\right] \cdot(\bw-\bw_h)
               -\sum_{e\in\mathcal{E}_h(\Gamma)}\int_e({\bf \mathcal{C}}\mathbf{e}(\bu_h) \bn_e +\kappa\bu_h)\cdot(\bw-\bw_h).}
 \end{array}
 \end{equation}
 Then, choosing $\bw_h$ as the Cl\'ement interpolant of $\bw$, that is $\bw_h\coloneqq I_h(\bw)$, the approximation properties of $I_h$ yield \cite{clement1975approximation}
 \begin{align}\label{eqn:clement}
 \left\| \bw-\bw_h \right\|_{0,K}\leq  c_1h_K\left\| \bw\right\|_{1,{\square(}K)}, \qquad 
 \left\| \bw-\bw_h \right\|_{0,e}\leq  c_2h_e\left\| \bw\right\|_{1,{\square(}e)},
 \end{align}
 where ${\square(}K)\coloneqq \cup\{K'\in \mathcal{T}_h: K'\cap K \neq \emptyset\}$ and $ {\square(}e) \coloneqq \cup\{K'\in \mathcal{T}_h: K'\cap e \neq \emptyset\}$.  In this way, applying the Cauchy--Schwarz inequality to each term (\ref{eqn:R}),
 and making use of (\ref{eqn:clement}) together with  the
 Lipschitz continuity of $F_{\bu}$ (cf. (\ref{eqn:lipschitz})), we obtain
 \begin{align*}
 \mathcal{R}_h(\bw-\bw_h) &\,\le\,
 \alpha c_1L_Fh_K\|\bu-\bu_h\|_{1,\Omega}\|\bw\|_{1,{\square(}K)} \\
& \quad
 + \,\, \widehat{C}\left\{ \sum_{K\in\mathcal{T}_h}\Theta_K^2\right\}^{1/2}
 \left\{\sum_{K\in\mathcal{T}_h}\|\bw\|_{1,{\square(}K)}^2+\sum_{e\in\mathcal{E}_h(\Omega)}\,
 \|\bw\|_{1,{\square(}e)}^2\right\}^{1/2}\,,
 \end{align*}
 where $\widehat{C}$ is a constant depending on $c_1$ and $c_2$ and $\Theta_K^2$ defined by (\ref{eqn:estimatorprim}). Additionally using the fact that the number of elements in
 ${\square(}K)$ and ${\square(}e)$ is bounded, we have
 $$\sum_{K\in\mathcal{T}_h}\|\bw\|_{1,{\square(}K)}^2\leq C_1\|\bw\|_{1,\Omega}^2\qquad{\rm and}\qquad\sum_{e\in\mathcal{E}_h(\Omega)}\|\bw\|_{1,{\square(}e)}^2\leq C_2\|\bw\|_{1,\Omega}^2,$$
 where $C_1,\,C_2$ are positive constants,  and using that $\alpha C_p L_F\leq 1/2$, it follows that $h_0\coloneqq 1/(2c_1\alpha L_F)$. Finally, substituting in (\ref{eqn:issus}) we conclude that
 $$\|\bu-\bu_h\|_{1,\Omega} \leq  C_{\mathrm{rel}}\,\Theta,$$
 where $C_{\mathrm{rel}}$ is independent of $h$.
 \end{proof}

The efficiency bound requires using a localisation technique based on element-bubble and edge-bubble functions. Given $K\in\mathcal{T}_h$ and $e\in\mathcal{E}(K)$, we define $\psi_K$ and $\psi_e$ the typical {element-} and edge-bubble functions \cite[eqs. (1.5)-(1.6)]{verfurth1999review}, which satisfy:
\begin{itemize}
\item [(i)] $\psi_K\in P_3(K)$, $\psi_K=0$ on $\partial K$, $\mathrm{supp}(\psi_K)\subseteq K$, and $0\leq\psi_K\leq 1$ in $K$,
\item [(ii)]  $\psi_e\in P_2(K)$, $\psi_e=0$ on $\partial K$, $\mathrm{supp}(\psi_e)\subseteq \omega_e$, and $0\leq\psi_e\leq 1$ in $\omega_e$,
\end{itemize}
where $\omega_e\coloneqq\cup\{K^{\prime}\in\mathcal{T}_h:\,\,e\in\mathcal{E}(K^{\prime})\}$.
Additional properties of $\psi_K$ and $\psi_e$ are collected in the following lemma (c.f. \cite[Lemma 1.3]{verfurth1994posteriori}, \cite[Section 3.4]{verfurth1996review}  or \cite[Section 4]{verfurth1999review}).
\begin{lemma}\label{lem:bubbles}
Given $k\in\mathbb{N}\cup\{0\}$, there exist positive constants $\gamma_1$, $\gamma_2$, $\gamma_3$, $\gamma_4$ and $\gamma_5$, depending only on k and the shape regularity of the
triangulations, such that for each $K\in\mathcal{T}_h$ and $e\in\mathcal{E}(K)$, there hold
\begin{equation}\label{eqn:bubble}
\begin{array}{rlll}
\gamma_1\left\| q\right\|_{0,K}^2&\leq& \left\| \psi_K^{1/2}q\right\|_{0,K}^2&\quad \forall\, q\in P_k(K),\\[1ex]
\left\| \psi_Kq\right\|_{1,K}&\leq& \gamma_2h_K^{-1}\left\| q\right\|_{0,K}&\quad \forall\, q\in P_k(K),\\[1ex]
\gamma_3\left\| p\right\|_{0,e}^2&\leq& \left\| \psi_e^{1/2}p\right\|_{0,e}^2&\quad \forall\, p\in P_k(e),\\[1ex]
\left\| \psi_ep\right\|_{1,\omega_e}&\leq& \gamma_4h_e^{-1/2}\left\| p\right\|_{0,e}&\quad \forall \,p\in P_k(e),\\[1ex]
\left\| \psi_ep\right\|_{0,\omega_e}&\leq &\gamma_5h_e^{1/2}\left\| p\right\|_{0,e}&\quad \forall\, p\in P_k(e).
\end{array}
\end{equation}
\end{lemma}

The efficiency (lower bound in (\ref{eqn:cotasprim})) is established with the help of the following lemma whose proof is a slight modification of \cite[Section 6]{verfurth1999review}.
\begin{lemma} \label{lem:effiprim}
    There exist constants $\eta_1,\eta_2,\eta_3>0$ and $C_{\mathrm{eff}}>0$, independent of $h$, but depending on $\gamma_i$ for $i\in\{1,\dots,5\}$ from Lemma~\ref{lem:bubbles} , such that for each $K\in\mathcal{T}_h$ there holds
\begin{align*}
h_K\left\|\alpha\boldsymbol{f}_{\bu_h} - \textrm{\bf div}({\bf \mathcal{C}}\mathbf{e}(\bu_h))\right\|_{0,K}&\leq \eta_1\| \bu-\bu_h\|_{0,K},\\ h_e^{1/2}\|[{\bf \mathcal{C}}\mathbf{e}(\bu_h)\cdot\bn_e]\|_{0,e}&\leq \eta_2\left\{\| \bu-\bu_h\|_{0,\omega_e}+\sum_{K\in \omega_e}h_K\|\bu-\bu_h\|_{0,K} \right\}, \\
h_e^{1/2}\|{\bf \mathcal{C}}\mathbf{e}(\bu_h)\cdot\bn_e\|_{0,e}&\leq \eta_3\| \bu-\bu_h\|_{0,K},
\end{align*}
where $\omega_e\coloneqq\cup\{K^{\prime}\in\mathcal{T}_h:\,\,e\in\mathcal{E}(K^{\prime})\}$. Further, it holds that
\begin{equation*}
    C_{\mathrm{eff}}\Theta \leq \|\bu-\bu_h\|_{1,\Omega}.
\end{equation*}

\end{lemma}
\begin{proof}
Using the properties of bubble functions, and letting $R_K(\bu_h)\coloneqq\alpha\boldsymbol{f}_{\bu_h} - \textrm{\bf div}({\bf \mathcal{C}}\mathbf{e}(\bu_h))$ we have
\begin{align*}
\left\| R_K(\bu_h)\right\|_{0,K}^2&\leq \gamma_1^{-1}\left\|\psi_K^{1/2}R_K(\bu_h)\right\|_{0,K}^2\\
&=\frac{1}{\gamma_1}\int_K \alpha\psi_KR_K(\bu_h)\{\boldsymbol{f}_{\bu_h}-\boldsymbol{f}_{\bu}\} -\frac{1}{\gamma_1}\int_K\psi_KR_K(\bu_h)\{ \textrm{\bf div}({\bf \mathcal{C}}\mathbf{e}(\bu_h)-{\bf \mathcal{C}}\mathbf{e}(\bu))\},\\
&=\frac{1}{\gamma_1}\int_K \alpha\psi_KR_K(\bu_h)\big\{\boldsymbol{f}_{\bu_h}-\boldsymbol{f}_{\bu}\} +\frac{1}{\gamma_1}\int_K({\bf \mathcal{C}}\mathbf{e}(\bu_h)-{\bf \mathcal{C}}\mathbf{e}(\bu))\cdot\nabla(\psi_KR_K(\bu_h)),\\
&\leq \frac{\alpha}{\gamma_1}\|R_K(\bu_h)\|_{0,K}\|\boldsymbol{f}_{\bu_h}-\boldsymbol{f}_{\bu}\|_{0,K} \\
& \qquad + \frac{\gamma_2}{\gamma_1h_K}\|{\bf \mathcal{C}}\mathbf{e}(\bu_h)-{\bf \mathcal{C}}\mathbf{e}(\bu)\|_{0,K}\|R_K(\bu_h)\|_{0,K},
\end{align*}
where, for the last inequality we used the inverse inequality. Next, we have
$$h_K\|R_K(\bu_h)\|_{0,K}\leq \alpha h_K \gamma_1^{-1}\|\boldsymbol{f}_{\bu_h}-\boldsymbol{f}_{\bu}\|_{0,K}+ \gamma_1^{-1}\gamma_2\|{\bf \mathcal{C}}\mathbf{e}(\bu_h)-{\bf \mathcal{C}}\mathbf{e}(\bu)\|_{0,K},$$
now, using (\ref{eqn:assumps}) and grouping terms, we conclude with $\eta_1>0$ independent of $h$, that 
$$h_K\left\|\alpha\boldsymbol{f}_{\bu_h} - \textrm{\bf div}({\bf \mathcal{C}}\mathbf{e}(\bu_h))\right\|_{0,K}\leq \eta_1\| \bu-\bu_h\|_{0,K}.$$
 We omit further details and repeated arguments used for the remaining inequalities.
\end{proof}

 \section{Solution strategy}\label{section:solution strategy}
 In this section we present how we propose to solve \eqref{eqn:problem}, which is mainly based on an accelerated proximal-point algorithm formulated in an infinite-dimensional setting. The \emph{proximal point method} \cite{rockafellar1976monotone} consists in adding a convex term to the original problem, so that given a solution $\bu^k$, we can rewrite the minimisation problem as
    \begin{equation}\label{eqn:pvar-time}
    \inf_{\bu\in \mathcal V} \alpha \mathcal D[\bu;R,T] + \mathcal S[\bu] + \frac{1}{\dt}\|\bu - \bu^k\|_{\bV},
    \end{equation}
    where the pseudo-timestep $\Delta t$ is a regularisation parameter, and we denote the solution as $\bu^{k+1}$. This scheme can be shown to converge under typically mild hypotheses \cite{kaplan1998proximal}, and its related Euler--Lagrange equations are given by
    \begin{equation}\label{eqn:problem-time}
      \frac{1}{\dt}\mathcal L \left(\bu^{k+1} - \bu^k\right) - \mathbf{div}\,(\mathcal{C}\epsilon(\bu^{k+1})) = \alpha \boldsymbol{f}_{\bu^{k+1}}, 
    \end{equation}
    where $\mathcal L$ is the operator induced by the inner product defining the norm in $\bV$. This can possibly induce a modification in the boundary conditions that is proportional to the velocity term $\dt^{-1}\left(\bu^{k+1} - \bu^k\right)$. This problem is still highly nonlinear because of the nonlinear term $f_{\bu^{k+1}}$, and so a common strategy is to treat it explicitly \cite{modersitzki2003numerical}, which yields the following semi-implicit (or IMEX) problem:
    \begin{equation}\label{eqn:problem-time-imex}
      \frac{1}{\dt}\mathcal L \bu^{k+1} - \mathbf{div}\,(\mathcal{C}\epsilon(\bu^{k+1})) = \frac{1}{\dt}\mathcal L \bu^k + \alpha \boldsymbol{f}_{\bu^k}. 
    \end{equation}
    This formulation has been shown to be stable under the timestep condition $\dt \approx 1/\alpha$ \cite{barnafi2018primal}. We now describe all the numerical choices that yield an efficient solution strategy of problem \eqref{eqn:problem}, based on the iterated solution of \eqref{eqn:problem-time-imex}. The use of a semi-implicit formulation yields a simpler problem to solve at each instant, but in turn it implies a CFL condition relating the pseudo-timestep and the spatial discretisation in order to have stability. Such a stability condition has not been established rigorously for this type of problem to the best of the authors knowledge.

    \paragraph{The proximal operator.} We consider two options for the operator $\mathcal L$. On the one hand, an $L^2$-norm in the regularised formulation \eqref{eqn:pvar-time}, which results in $\mathcal L=I$ the identity operator. On the other hand, an $H^1$-norm which results in $\mathcal L=-\Delta + I$, which gives rise to Sobolev gradient stabilisation in the context of Levemberg--Marquardt methods \cite{kazemi2012levenberg}. We then note that \eqref{eqn:problem-time-imex} written in terms of the increment $\delta \bu^{k+1}$ results in 
\begin{equation*}
  \left(\frac{1}{\dt}\mathcal L  - \mathbf{div}\,\mathcal{C}\epsilon\right)\delta \bu^{k+1} = \mathbf{div}\,(\mathcal{C}\epsilon(\bu^{k})) + \alpha \boldsymbol{f}_{\bu^k}, 
\end{equation*}
where the right-hand side $\mathbf{div}\,(\mathcal{C}\epsilon(\bu^{k})) + \alpha \boldsymbol{f}_{\bu^k}$ is the residual of the original problem \eqref{eqn:problem}. This shows that this method can also be interpreted as a Levemberg--Marquardt type of algorithm \cite{nocedal1999numerical}. The stabilisation matrix in this case is the operator $\mathcal L$, and the gradient $\nabla_{\bu} \boldsymbol{f}_{\bu}$ is neglected from the complete Jacobian, given by
$$ \mathcal J(\bu) = d^2\mathcal S + \alpha(d^2\mathcal D(\bu)) = -\mathbf{div}\,\mathcal C\epsilon + \alpha\underbrace{\left(\nabla T(\ten I + \bu)\otimes \nabla T(\ten I + \bu) + (T(\ten I + \bu) - R) \mathbf{H}T(\ten I + \bu)\right)}_{\nabla_{\bu}\boldsymbol{f}_{\bu}}, $$
where $\mathbf{H}T$ stands for the Hessian of $T$ and $\ten I$ is the identity function.

\paragraph{Discrete spaces.} The left-hand side operator in \eqref{eqn:problem-time-imex} consists in a linear elasticity operator plus an identity, so we can use $\bH^1$-conforming \acp{fe}. In particular, we consider a vector-valued Lagrangian \ac{fe} space. For a quadrilateral ($d=2$) or hexahaedral ($d=3$) discretization $\mathcal T_h$ of $\Omega$, this space is given by 
    $$ Q_h^k = \{ \bv \in C(\Omega, \bbR^d): \bv|_K \in \bbQ^d_k(K)\quad\forall K \in \mathcal T_h\}, $$
    where $\bbQ^d_k=\bbQ_k \otimes \ldots \otimes \bbQ_k$, and $\bbQ_k$ is the space of $d$-variate tensor-product polynomials of partial degree at most $k$ with respect to each variable, for each of the $d$ components of the vector-valued field.

\paragraph{Solvers in use.} To solve the linear system arising from \eqref{eqn:problem-time-imex}, we use an efficient sparse LU factorisation, which we compute once and then reuse it throughout all time iterations.

\subsection{Time acceleration}\label{section:aa}
Note that we can rewrite \eqref{eqn:problem-time-imex} as 
    \begin{equation}    
        G(\bu^{k+1}) = F(\bu^k),
    \end{equation}
    and, as $G$ is invertible, the solution can be characterised as a fixed point of the operator $T = G^{-1} \circ F$. This motivates using \ac{aa} as a fixed-point acceleration algorithm. The acceleration of time iterations to compute steady state computations has already been successfully used in nonlinear poroelasticity \cite{barnafi2024fully} using \ac{aa}, so we adopt the same strategy. \ac{aa} of depth $m$ consists in the following: consider a fixed-point iteration function $g$, whose iterations are given by $x^{k+1} = g(x^k)$. Then, at each iteration $k$:
    \begin{enumerate}
        \item Define the matrix $F_k = [f_{k-m}, \hdots, f_k]$, where $f_\ell = g(x^\ell) - x^\ell$.
        \item Determine $\alpha = (\alpha_0, \hdots, \alpha_k)$ that solves
                $$ \min_{\sum_j \alpha_j = 1} \|F_k\alpha \|^2. $$
        \item Compute the accelerated update 
                \begin{equation}\label{eqn:aa}
                    x^{k+1} = \sum^k_{i=0} \alpha_i g(x^{k-m+i}).
                \end{equation}
    \end{enumerate}
    We will denote, for each iterate $\bu^k$, as $AA_m(\bu^k)$ the accelerated solution obtained with \eqref{eqn:aa}, where the fixed-point map $g$ is given by our solution map $G^{-1}\circ F$. 

\subsection{The adaptive solver} 

The steps in our adaptive solver are summarized in Algorithm~\ref{alg:amr}. The algorithm leverages non-conforming forest-of-octrees meshes; see, e.g., \cite{badia2020generic}. 
Forest-of-octrees meshes can be seen as a two-level decomposition of the computational domain (typically an square or cube in the case of \ac{dir}) referred to as macro and micro level, respectively. The macro level is a suitable {\em conforming} partition $\mathcal{C}_h$ of into quadrilateral ($d=2$) or hexahedral cells ($d=3$). This mesh, which may be generated using for instance an unstructured mesh generator,  is referred to as the coarse mesh. At the micro level, each of the cells of $\mathcal{C}_h$ becomes the root of an adaptive octree with cells that can be recursively and dynamically refined or coarsened using the so-called $1:2^d$ uniform partition rule. If a cell is marked for refinement, then it is split into $2^d$ children cells by subdividing all parent cell edges. If all children cells of a parent cell are marked for coarsening, then they are collapsed into the parent cell. The union of all leaf cells in this hierarchy forms the decomposition of the domain at the micro level.

\begin{algorithm}
    \caption{\ac{aa}-\ac{amr} solution strategy ($\alpha$, $m$, $N^0_\texttt{ref}$, $N_\texttt{ref}$, $\theta_\texttt{coarsen}$, $\theta_\texttt{refine}$)}\label{alg:amr}
    \begin{algorithmic}[1]
        \STATE{Refine a one-element mesh $N^0_\texttt{ref}$ times to get initial mesh $M^0$}
        \STATE{Initialise solution vector $\bu$ in $M^0$}
        \FOR{$k$ in  $1,\hdots, N_\texttt{ref}$}
            \STATE{Interpolate solution $\bu$ to current mesh $\bu^0 = \Pi_{M^{k-1}} \bu$}
            \WHILE{Solution not converged}
                \STATE{Set $\tilde{\bu}^{k+1}$ the solution of \eqref{eqn:problem-time} given $\bu^k$}
                \STATE{Set $\bu^{k+1} = AA_m(\tilde{\bu}^{k+1})$ the accelerated solution}
            \ENDWHILE
            \STATE{Set solution $\bu$ as the last iteration}
            \STATE{Compute error estimate \eqref{eqn:estimatorprim}}
            \STATE{Refine $\theta_\texttt{refine}$ fraction of elements with largest estimator in $M^{k-1}$}
            \STATE{Coarsen $\theta_\texttt{coarsen}$ fraction of elements with smallest estimator in $M^{k-1}$}
            \STATE{Set new mesh $M^k$ and adapt discrete spaces}
        \ENDFOR
        \RETURN $\bu$
    \end{algorithmic}
\end{algorithm}

One important consideration is that \ac{amr} with forest-of-octrees allows for coarsening only if there is an initial mesh hierarchy to be coarsened. Because of this, we consider for all problems an initial coarse mesh $\mathcal{C}_h$ with only one coarse element (where the images fit), and perform some initial uniform refinements that give us such initial hierarchy. This is relevant as images have in many cases a dark background where accuracy is not important, and thus we expect our error estimator to detect this and coarsen such areas. Because of these considerations, our proposed algorithm consists in the following steps: (i) build an initial mesh by uniformly refining a single-element coarse mesh, (ii) on each mesh solve \ac{dir} \eqref{eqn:problem} using scheme \eqref{eqn:problem-time}, (iii) after computing a solution for a given mesh, a $\theta_\text{refine}$ percentage of elements are refined and a $\theta_\text{coarsen}$ percentage of elements are coarsened, and (iv) stop after a given number of mesh adaptations has been performed. We stress that this approach resembles the octree-based approach from \cite[Sect. 3]{haber2007octree}, but with the following key differences: (i)~we use a \ac{fem} formulation (instead of finite differences), (ii)~we guide adaptivity using a theoretically derived residual-based error indicator, and (iii)~we combine \ac{amr} with \ac{aa} to speed-up the convergence of the pseudo-time formulation at each mesh level in the hierarchy. 

Given that our target problem is \eqref{eqn:problem}, the algorithm performns adaptivity only after the pseudo-time simulation based on \eqref{eqn:problem-time-imex} has reached a stationary state. We acknowledge that further optimisation can be obtained by performing adaptivity instead every fixed number of timesteps, but as the number of timesteps is highly unpredictable, we preferred not to pursuit this strategy. In addition, the use of \ac{aa} for the pseudo-time iterations would be otherwise incompatible with adaptivity due to the change of dimensions between different adaptivity steps. 

\section{Numerical tests}\label{section:numerical tests}

 In this section we present numerical examples to illustrate the performance of the proposed adaptive \ac{dir} solver.  The realisation of this solver is conducted using the tools provided by the open source scientific software packages in the \texttt{Gridap} ecosystem \cite{badia22,verdugo22}.
 The sparse linear systems were solved using UMFPACK, as provided by Julia. We used the \texttt{GridapP4est.jl}~\cite{GridapP4est2025} Julia package in order to handle forest-of-octrees meshes (including facet integration on non-conforming interfaces as per required by the computation of the a posteriori error estimator) and \ac{fe} space constraints. This package, built upon the \texttt{p4est} meshing engine~\cite{burstedde2011p4est}, is endowed with Morton space-filling curves, and it provides high-performance and low-memory footprint algorithms to handle forest-of-octrees.
 All numerical tests were performed on a supercomputer node equipped with Intel Xeon Platinum 8274 CPU cores, with Julia 1.10.4 and IEEE double precision. We used \texttt{-O3} as the optimisation  flag for the Julia compiler. 

\begin{figure}[t!]
    \centering
    \begin{subfigure}{0.45\textwidth}
        \centering
        \includegraphics[width=0.8\textwidth]{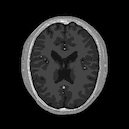}
        \subcaption{Brain reference}
    \end{subfigure}
    \begin{subfigure}{0.45\textwidth}
        \centering
        \includegraphics[width=0.8\textwidth]{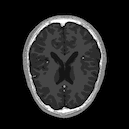}
        \subcaption{Brain target}
    \end{subfigure}

    \begin{subfigure}{0.45\textwidth}
        \centering
        \includegraphics[width=0.8\textwidth]{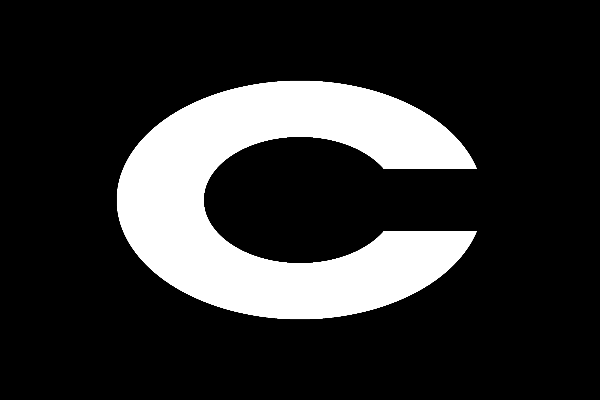}
        \subcaption{OC reference}
    \end{subfigure}
    \begin{subfigure}{0.45\textwidth}
        \centering
        \includegraphics[width=0.8\textwidth]{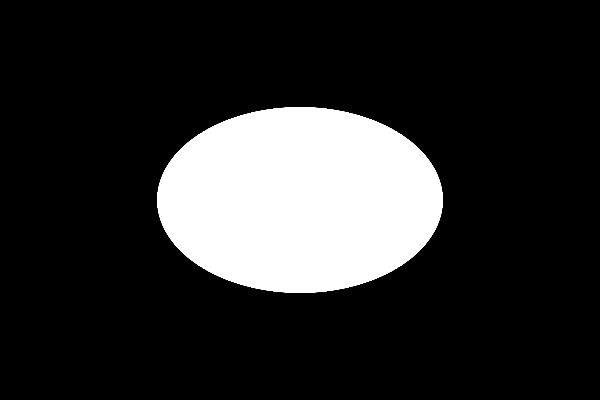}
        \subcaption{OC target}
    \end{subfigure}
    \caption{Reference (left) and target (right) images used in the brain (top) and OC (bottom) tests. }\label{fig:images}
\end{figure}

 For simplicity, we will test our algorithm in two brain images, obtained from the BrainWeb Database \cite{cocosco1997brainweb,collins1998design,kwan1996extensible,kwan1999mri} and both of $129\times 129$ pixels, and on a standard benchmark test known as the OC images, which we stored as $600\times 400$ pixels images. We show both image pairs in Figure~\ref{fig:images}. The brain test does not require large deformations to take place, but is nonetheless a challenging problem as brain images have regions with large contrasts, mainly near the cortex. The OC test is instead much more challenging, as the registration we look for require very large deformations and is guided only through the discrete gradient of a field that is fundamentally discontinuous. Unless otherwise stated, all images will be pre-processed with a Gaussian kernel using a variance parameter of $\sigma=1$ to avoid having discontinuous derivatives, and then interpolated globally using linear B-Splines \cite{de1978practical}, so that we can compute the gradient of the template image as required by the $\nabla T$ term. 

 A fundamental issue that we do not address in detail in this work is that whenever we use elements that have more than one image pixel inside, there could be additional numerical integration errors involved in our computations, in particular for those terms involving images. In preliminary studies we have observed this to have a non-negligible influence on the performance of the solvers, and thus we show a sensitivity analysis regarding this issue in Section~\ref{appendix:integration} using first order Lagrangian \acp{fe}. In the experiments, we chose a quadrature order of 6 for all integrals involving images (see the aforementioned section for a formal definition of quadrature order, and the particular kind of quadratures that we used in this work). We envision that the use of adaptive integration rules for $n$-cubes (such as, e.g., \cite{genz1980remarks,johnson2018algorithm}) for terms involving images may lead to a better trade-off among integration accuracy and computational effort (see also \cite{bull1995parallel}). Because of this issue, we have performed all of our tests using only first order elements, as higher order approximations would require further tuning the quadrature rules.

\subsection{Convergence verification against manufactured   solutions}

In order to confirm the accuracy of the proposed schemes we perform two simple tests of convergence with exact solutions. We consider the unit square domain $\Omega = (0,1)^2$, use synthetic images 
\[ R(\bx) = |\bx - (0.2,0.2)^{\tt t}|^2 , \quad T(\bx) = |\bx - (0.8,0.8)^{\tt t}|^2, \]
and use the following smooth and 
non-smooth displacement solution 
\[ 
\bu_{\mathrm{ex}} = \frac{1}{10}\begin{pmatrix}\bigl[-\sin(\pi x)+\frac{1}{\lambda}\cos(\pi x)\bigr]\sin(\pi y)+\frac{4}{\pi^2}\\
\bigl[-\cos(\pi x)+\frac{1}{\lambda}\sin(\pi x)\bigr]\cos( \pi y)\end{pmatrix}, \quad \text{and} \quad 
 \bu_{\mathrm{ex}} = \frac{r^\beta}{10}\begin{pmatrix} \cos(\beta \theta)\\ \sin(\beta\theta)
\end{pmatrix}, 
\]
respectively. 
For the non-smooth case 
$r = \sqrt{x^2+y^2}$, $\theta = \mathrm{atan2}(y,x)$ are the polar coordinates and $\beta = \frac23$ is a regularity index that yields $\bu \in \bH^{1+\beta}(\Omega)$, and therefore we expect (under uniform refinement) a suboptimal convergence of $O(h^\beta)$. We consider the parameters $\Delta t = \alpha = E =1$, $\nu = \frac14$, and $\beta = \frac23$. For the smooth case we take $\kappa =0.5$ and for the non-smooth case we consider the pure-traction boundary condition (setting the boundary stiffness parameter $\kappa=0$). Note that the exact  solutions above do not induce zero traction boundary conditions for planar elasticity, so we need to also manufacture an exact traction imposed weakly in the formulation. Similarly, the manufactured solution may have a component of rigid body motions and so we also include on the right-hand side a contribution taking into account this part of the kernel. Likewise, an additional contribution is required as manufactured load on the right-hand side of the momentum equation, as well as in the element contribution to the volume error estimator and in the edge/boundary contribution (since the manufactured normal stress is non-homogeneous). In particular, this gives an estimator of the form 
\begin{align*}
   \Theta_K^2 & \coloneqq h_K^2\|\frac{1}{\Delta t} \bu_h  - \boldsymbol{f}_{\mathrm{ex}} + \boldsymbol{g}_{\mathrm{ex}} - \alpha\boldsymbol{f}_{\bu_h}\! - \textbf{div}({\bf \mathcal{C}}\mathbf{e}(\bu_h))\|_{0,K}^2 \\
   &\quad +\!\!\!\sum_{e\in\mathcal{E}(K)\cap\mathcal{E}_h(\Omega)}\!\!\! h_e\|[({\bf \mathcal{C}}\mathbf{e}(\bu_h)-\bsigma_{\mathrm{ex}})\bn_e+\kappa(\bu_h-\bu_{\mathrm{ex}})]\|_{0,e}^2 \\&\quad 
       + \!\!\! \sum_{e\in\mathcal{E}(K)\cap\mathcal{E}_h(\Gamma)}\!\!\!h_e\|({\bf \mathcal{C}}\mathbf{e}(\bu_h)-\bsigma_{\mathrm{ex}})\bn_e+\kappa(\bu_h-\bu_{\mathrm{ex}})\|_{0,e}^2,  
\end{align*}
with $\boldsymbol{f}_{\mathrm{ex}} = - \mathbf{div}\,\bsigma_{\mathrm{ex}}$ and $\boldsymbol{g}_{\mathrm{ex}} = \alpha \{T(\bx + \bu_{\mathrm{ex}})-R\}\nabla T(\bx + \bu_{\mathrm{ex}})$.

The numerical results of the uniform refinement tests are reported in Table~\ref{table:convergence-smooth} and \ref{table:convergence-non-smooth} for the smooth and non-smooth solutions, respectively. 
We construct six levels of uniform mesh refinement of the domain, on which we compute approximate solutions and the associated errors in the norm $|\bv|_{1,\Omega} = \|\boldsymbol{\varepsilon}(\bv)\|_{0,\Omega}$. Convergence rates are calculated as usual: 
\[
\text{rate} =\log({e}/\widehat{{e}})[\log(h/\widehat{h})]^{-1}\,,\]
where ${e}$ and $\widehat{{e}}$ denote errors produced on two consecutive meshes of sizes $h$ and $\widehat{h}$, respectively. For the smooth case we see optimal convergence of order $O(h^{k+1})$ and bounded effectivity indexes independently of the refinement level. For the non-smooth case we observe the expected sub-optimal convergence under a uniform mesh refinement.

\begin{table}[htbp]
\setlength{\tabcolsep}{4.pt}
    \centering
    \begin{tabular}{r c | c c c}
        \toprule \acp{dof}   & $h$ & $|\bu-\bu_h|_{1,\Omega}$ & \texttt{rate} & $\texttt{eff}(\Theta)$ \\ \midrule
        \multicolumn{5}{c}{Uniform mesh refinement, with $k=1$}\\
          \midrule
     21 & 0.7071 & 6.48e-02 & $\star$ & 0.234 \\
     53 & 0.3536 & 3.29e-02 & 0.977 & 0.231 \\
    165 & 0.1768 & 1.66e-02 & 0.991 & 0.214 \\
    581 & 0.0884 & 8.30e-03 & 0.997 & 0.206 \\
   2181 & 0.0442 & 4.15e-03 & 0.999 & 0.202 \\
   8453 & 0.0221 & 2.08e-03 & 1.000 & 0.200 \\
 \midrule 
\multicolumn{5}{c}{Uniform mesh refinement, with $k=2$}\\
          \midrule
     53 & 0.7071 & 1.32e-02 & $\star$ & 0.062 \\
    165 & 0.3536 & 3.34e-03 & 1.985 & 0.063 \\
    581 & 0.1768 & 8.39e-04 & 1.995 & 0.063 \\
   2181 & 0.0884 & 2.10e-04 & 1.996 & 0.064 \\
   8453 & 0.0442 & 5.33e-05 & 1.979 & 0.065 \\
  33285 & 0.0221 & 1.51e-05 & 1.946 & 0.077\\
     \bottomrule
    \end{tabular}
    \caption{Convergence tests against smooth manufactured solutions. Error history of the method for two polynomial degrees and effectivity index associated with the \emph{a posteriori} error estimator on uniform mesh refinement.}
    \label{table:convergence-smooth}
\end{table}

On the other hand, the numerical results for the adaptive mesh refinement case are reported in Table \ref{table:convergence-non-smooth}. We take (for the two polynomial degrees tested here) the same refinement fraction of 15\%. 
For the experimental convergence rates of the adaptive case we use the alternative form 
\[\text{rate} = -2\log(e/\widehat{e})[\log(\text{\acp{dof}}/\widehat{\text{ \acp{dof}}})]^{-1}.\]
One can see in Table \ref{table:convergence-non-smooth} that the optimal convergence is attained under adaptive mesh refinement guided by the \emph{a posteriori} error estimator. Again, the effectivity index remains bounded in all cases. Also, we can readily see that the same level of energy error is reached with the adaptive case using roughly 10\% of the number of \acp{dof} required in the uniform mesh refinement. This is consistent in both first and second order schemes. 

\begin{table}[htbp]
\setlength{\tabcolsep}{3.5pt}
\begin{minipage}{0.495\textwidth}
\centering    \begin{tabular}{r c | c c c}
    \toprule \acp{dof}   & $h$ & $|\bu-\bu_h|_{1,\Omega}$ & \texttt{rate} & $\texttt{eff}(\Theta)$ \\ \midrule
        \multicolumn{5}{c}{Uniform mesh refinement, with $k=1$}\\
          \midrule
           21 & 0.7071 & 1.30e-02 & $\star$ & 0.178 \\
     53 & 0.3536 & 8.60e-03 & 0.599 & 0.215 \\
    165 & 0.1768 & 5.61e-03 & 0.616 & 0.228 \\
    581 & 0.0884 & 3.63e-03 & 0.628 & 0.233 \\
   2181 & 0.0442 & 2.34e-03 & 0.634 & 0.237 \\
   8453 & 0.0221 & 1.50e-03 & 0.638 & 0.240\\ 
     33285 & 0.0110 & 9.65e-04 & 0.640 & 0.244\\
     \midrule 
     \multicolumn{5}{c}{Adaptive mesh refinement, with $k=1$}\\
          \midrule
    53 & 0.7071 & 8.60e-03 & $\star$ & 0.215 \\
     71 & 0.3536 & 6.20e-03 & 2.240 & 0.229\\
     97 & 0.1768 & 4.65e-03 & 1.848 & 0.234\\
    137 & 0.0884 & 3.64e-03 & 1.408 & 0.237\\
    187 & 0.0442 & 2.87e-03 & 1.530 & 0.242\\
    259 & 0.0221 & 2.37e-03 & 1.180 & 0.239\\
    367 & 0.0110 & 1.88e-03 & 1.334 & 0.227\\
    511 & 0.0055 & 1.55e-03 & 1.156 & 0.221\\
    \bottomrule 
\end{tabular}\end{minipage}\begin{minipage}{0.495\textwidth}
    \centering    \begin{tabular}{r c | c c c}
        \toprule \acp{dof}   & $h$ & $|\bu-\bu_h|_{1,\Omega}$ & \texttt{rate} & $\texttt{eff}(\Theta)$ \\ \midrule
\multicolumn{5}{c}{Uniform mesh refinement, with $k=2$}\\
          \midrule
          53 & 0.7071 & 6.09e-03 & $\star$ & 0.111\\
    165 & 0.3536 & 4.13e-03 & 0.558 & 0.184\\
    581 & 0.1768 & 2.80e-03 & 0.563 & 0.246\\
   2181 & 0.0884 & 1.88e-03 & 0.571 & 0.284\\
   8453 & 0.0442 & 1.26e-03 & 0.580 & 0.308\\
  33285 & 0.0221 & 8.39e-04 & 0.588 & 0.327\\
 132101 & 0.0110 & 5.56e-04 & 0.594 & 0.345\\
 \midrule 
\multicolumn{5}{c}{Adaptive mesh refinement, with $k=2$}\\
          \midrule
     165 & 0.7071 & 4.13e-03 & $\star$ & 0.184\\
    231 & 0.3536 & 2.83e-03 & 2.253 & 0.235\\
    327 & 0.1768 & 1.96e-03 & 2.100 & 0.259\\
    477 & 0.0884 & 1.37e-03 & 1.927 & 0.267\\
    699 & 0.0442 & 9.34e-04 & 1.987 & 0.269\\
    963 & 0.0221 & 6.53e-04 & 2.231 & 0.293\\
   1373 & 0.0110 & 4.32e-04 & 2.335 & 0.294\\
   1981 & 0.0055 & 2.77e-04 & 2.423 & 0.284\\
     \bottomrule
    \end{tabular}\end{minipage}
    \caption{Convergence tests. Error history of the method for two polynomial degrees, using a non-smooth manufactured solution, and effectivity index associated with the \emph{a posteriori} error estimator on uniform and adaptive mesh refinement.}
    \label{table:convergence-non-smooth}
\end{table}

To exemplify the performance of the method in the non-smooth solution regime, we plot in Figure \ref{fig:convergence} the approximate solution (displacement magnitude warped) in a coarse adapted mesh, the synthetic images (in the undeformed mesh), as well as depictions of the adapted meshes after several steps of refinement, which indicate the expected agglomeration of elements near the origin (where the singularity is). 

\begin{figure}[htbp]
    \centering
    \includegraphics[width=0.324\linewidth]{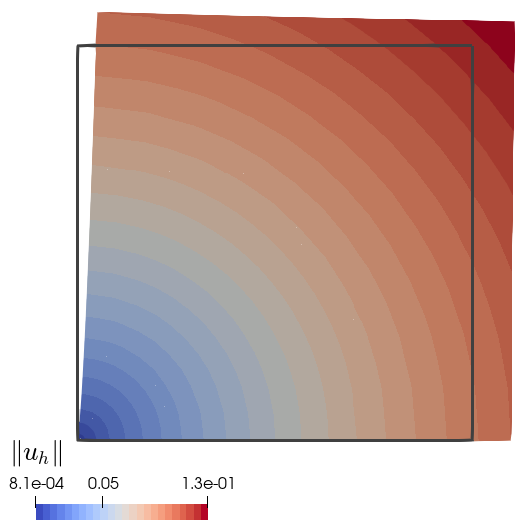}
    \includegraphics[width=0.324\linewidth]{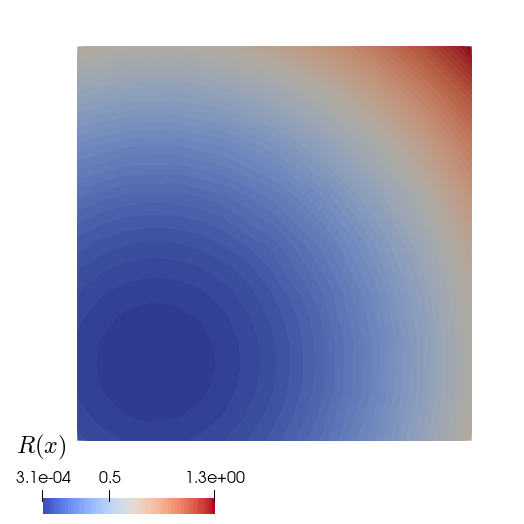}
    \includegraphics[width=0.324\linewidth]{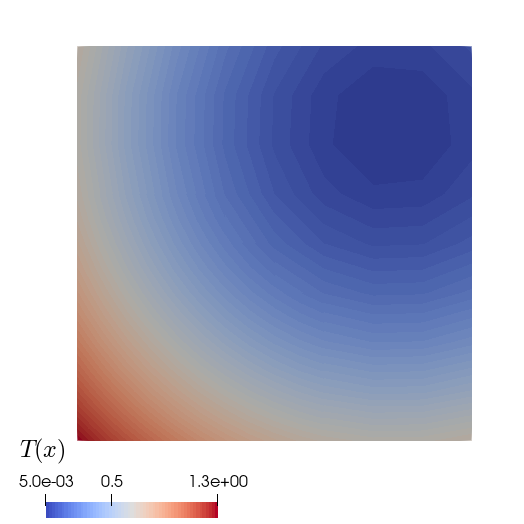}\\
    \includegraphics[width=0.324\linewidth]{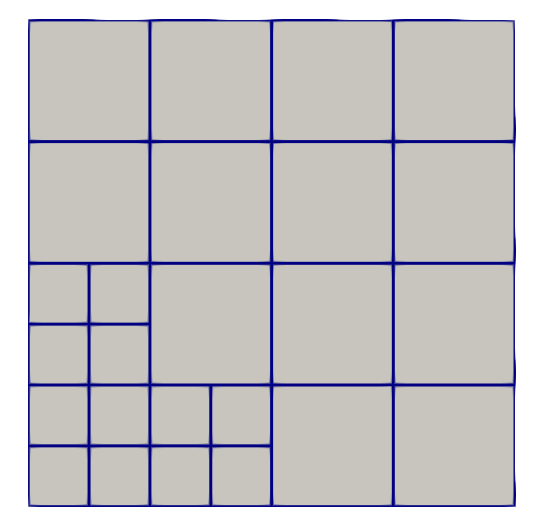}
    \includegraphics[width=0.324\linewidth]{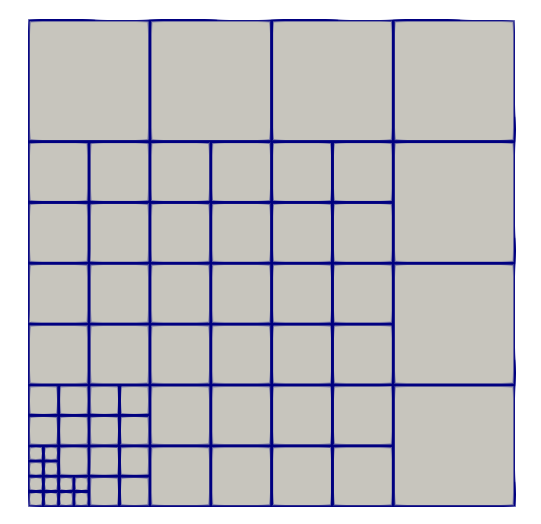}
    \includegraphics[width=0.324\linewidth]{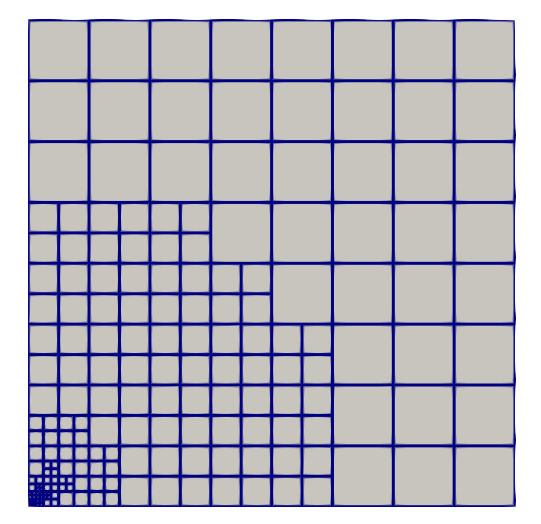}
    \caption{Convergence test  against a singular solution. Approximate displacement magnitude (warped), synthetic reference and target images all on the final adapted mesh (top row), and sample of adaptively refined meshes guided by the \emph{a posteriori} error estimator $\Theta$ after 3,5,7 refinement steps (bottom row).}
    \label{fig:convergence}
\end{figure}

\subsection{Anderson acceleration for \ac{dir}}
In this section we study the performance of the time acceleration method through \ac{aa} described in Section~\ref{section:aa}. We do this for both the brain and OC images, to evaluate the technique under small and large deformations.

We show the similarity measure, total number of iterations and elapsed time for the brain images registration in Table~\ref{table:aa-brain}. For this test, we used $\alpha=10^4$ and $\Delta t=10^{-5}$. The mesh resolution was set such that there is a single pixel image per each cell. To avoid ambiguities regarding convergence, we used the relative Euclidean vector norm of the stationary residual as a convergence criterion, with a tolerance of $10^{-4}$. 
We highlight that better results were obtained using an $L^2$ stabilisation in the pseudo-time terms from \eqref{eqn:problem-time}, i.e. $\mathcal L=I$ the identity operator. In this case, \ac{aa} is extremely convenient, as it provides a significant reduction in the number of iterations that increases as the depth parameter $m$ increases, except for the case $m=5$, which is still more convenient than the non-accelerated strategy. We registered the results until a wider depth was not convenient anymore, as iterations started increasing again, which typically happens because the conditioning of the related least squares problem starts deteriorating. For the largest depth parameter, using \ac{aa} can yield time accelerations up to a factor of 22 with respect to a non-accelerated approach. 

\begin{table}[htbp]
    \centering
    \begin{tabular}{l | r r r }
        \toprule Scheme  & Similarity & Iterations & elapsed time (s) \\ \midrule
        No accel & 0.0412 & 2,289      &   3,520.81     \\
        AA(2)    & 0.0412 & 382        &   651.43      \\
        AA(5)    & 0.0412 & 581        &   997.12      \\
        AA(10)   & 0.0412 & 203        &   332.04      \\ 
        AA(20)   & 0.0412 & 99         &   157.61      \\ \bottomrule
    \end{tabular}
    \caption{Acceleration test, brain images: Similarity measure, iteration count and elapsed time required for convergence for varying \ac{aa} depth, given by the parameter $m$.}
    \label{table:aa-brain}
\end{table}

For the OC images test we consider two cases for the values of $\alpha$ given by $10^4$ and $10^5$, and as the images have a sharp discontinuity, we considered $\sigma=4$. We set the size of the pseudo time-step to $\Delta t=10^{-2}$ and  $\Delta t=10^{-3}$, respectively. The mesh resolution was set in order to have $4 \times 4$ image pixels  per each cell. As with the brain, we used the relative Euclidean vector norm of the stationary residual as a convergence criterion for this test, but with a tolerance of $10^{-2}$. For this problem, solutions are obtained using $\mathcal L=I - \Delta$ from \eqref{eqn:problem-time}, as using only the identity yielded diverging iterations in our preliminary tests. This benchmark requires far larger deformations to obtain a satisfactory registration, which we depict by showing the warped target images next to the reference image for both values of $\alpha$ in Figure~\ref{fig:oc-small-large}. For the smaller $\alpha$ case, we show the results in Table~\ref{table:aa-oc-small}, where the results obtained are very similar to those of the brain. Acceleration is very convenient as it drastically reduces the number of iterations with minimal overhead per iteration. This yields elapsed time reductions of up to a factor of 3.7. In spite of this, we note that, as shown in Table~\ref{table:aa-oc-large}, acceleration is ineffective when larger displacements are involved, so that all accelerated iterations achieved the maximum number of iterations allowed ($10,000$).

\begin{figure}[htbp]
    \centering
    \newcommand{\wid}{0.32\textwidth}
    \begin{subfigure}{\wid}
        \centering
        \includegraphics[width=\textwidth]{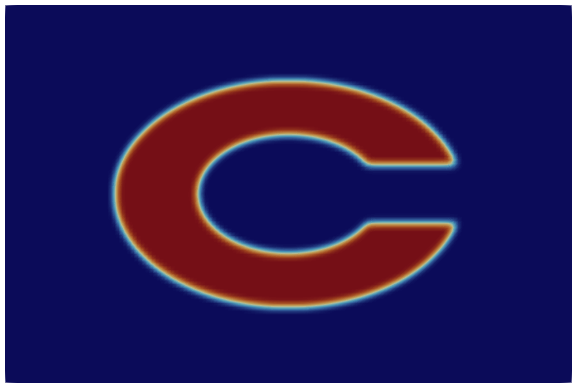}
        \caption{Reference}
    \end{subfigure}
    \begin{subfigure}{\wid}
        \centering
        \includegraphics[width=\textwidth]{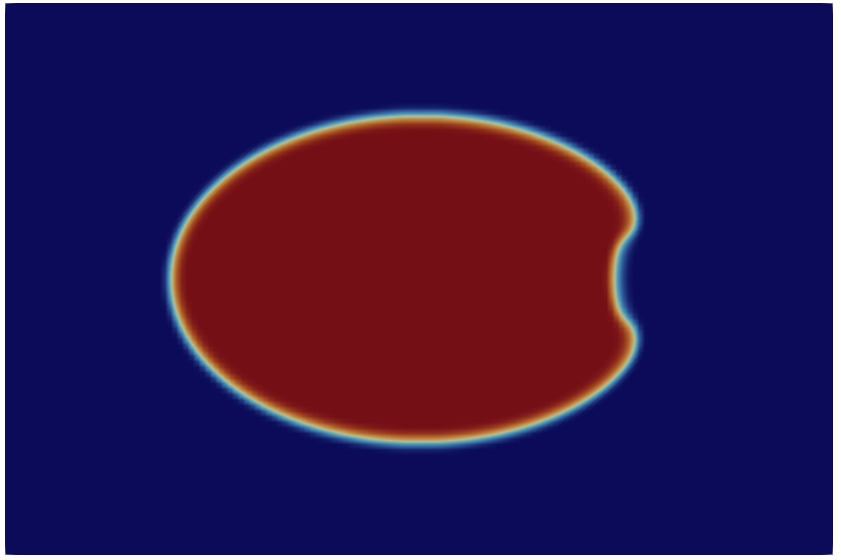}
        \caption{Small $\alpha$}
    \end{subfigure}
    \begin{subfigure}{\wid}
        \centering
        \includegraphics[width=\textwidth]{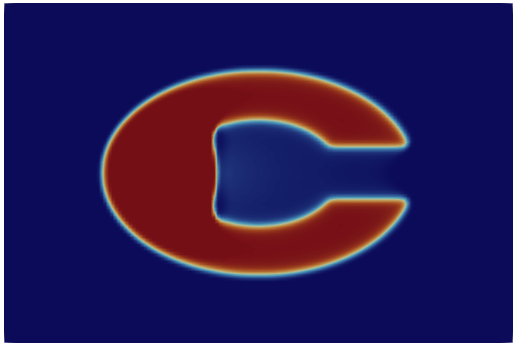}
        \caption{Large $\alpha$}
    \end{subfigure}
    \caption{OC test solutions for small and large similarity parameter ($10^4$ and $10^5$ respectively).}
    \label{fig:oc-small-large}
\end{figure}

\begin{table}[htbp]
\setlength{\tabcolsep}{4.pt}
    \centering
    \begin{subfigure}{0.49\textwidth}
        \centering
        \begin{tabular}{l | r r r }
            \toprule 
            Scheme   & Similarity & Iterations & time (s) \\ \midrule
            No accel & 0.263 & 113        &      78.43  \\
            AA(2)    & 0.263 & 39        &      28.58   \\
            AA(5)    & 0.263 & 37         &     27.03    \\
            AA(10)   & 0.264 & 28         &    21.02     \\
            AA(20)   & 0.264 & 28         &    21.93     \\ \bottomrule
        \end{tabular}
        \caption{Small $\alpha$}
        \label{table:aa-oc-small}
    \end{subfigure}
    \begin{subfigure}{0.49\textwidth}
        \centering
        \begin{tabular}{l | r r r }
            \toprule 
            Scheme   & Similarity & Iterations &  time (s) \\ \midrule
            No accel & 0.0846 & 8,927       & 6,123      \\
            AA(2)    & -- & --         & --           \\
            AA(5)    & --  & --         & --           \\
            AA(10)   & -- & --         & --           \\
            AA(20)   & -- & --         & --           \\ \bottomrule
        \end{tabular}
        \caption{Large $\alpha$}
        \label{table:aa-oc-large}
    \end{subfigure}
    \caption{Acceleration test, OC images: Similarity measure, iteration count and elapsed time required for convergence for varying \ac{aa} depth, given by the parameter $m$.}
\end{table}

\subsection{Adaptivity performance on brain images}
Starting from a mesh of one element, we considered 4 initial uniform refinements, and perform 5 adaptive mesh refinements with refine and coarse parameters given by $\theta_\texttt{refinement}= 0.4$ and $\theta_\texttt{coarsen} = 0.2$ respectively. We considered a tolerance of $10^{-4}$ for the velocity $\frac{\bu^{k+1} - \bu^k}{\Delta t}$ in the reference (non-adapted) case, and a tolerance of $10^{-2}$ for all the adaptive cases. We performed all tests using \ac{aa} with a depth parameter of $m=10$. We have seen that considering an equally precise solution in all levels of the adaptive solver leads to over-solving, and thus the resulting scheme that we consider resembles an inexact-Newton procedure.  To obtain a solution with better similarity than the one in Table~\ref{table:aa-brain} we used $\alpha=10^5$, and set $\Delta t=10^{-6}$ as larger timesteps resulted in non-convergence (in the form of oscillating iterations).

\begin{figure}[htbp]
    \newcommand{\wid}{0.16\textwidth}
    \centering
    \begin{subfigure}{\wid}
        \centering
        \includegraphics[width=\textwidth]{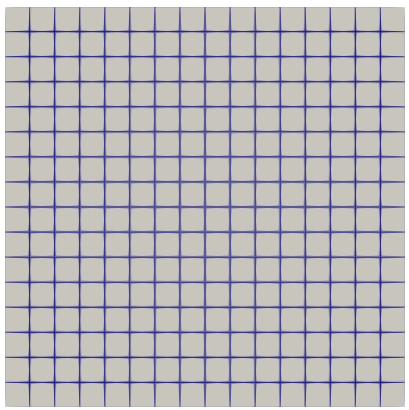}
    \end{subfigure}
    \begin{subfigure}{\wid}
        \centering
        \includegraphics[width=\textwidth]{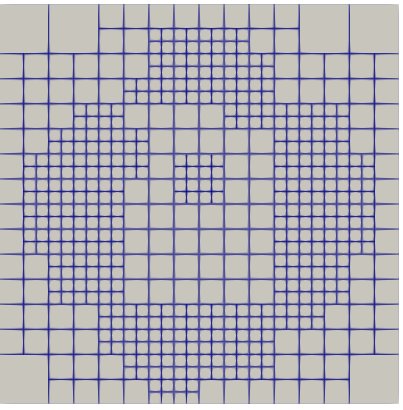}
    \end{subfigure}
    \begin{subfigure}{\wid}
        \centering
        \includegraphics[width=\textwidth]{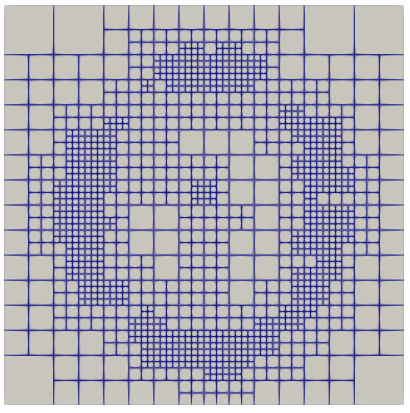}
    \end{subfigure}
    \begin{subfigure}{\wid}
        \centering
        \includegraphics[width=\textwidth]{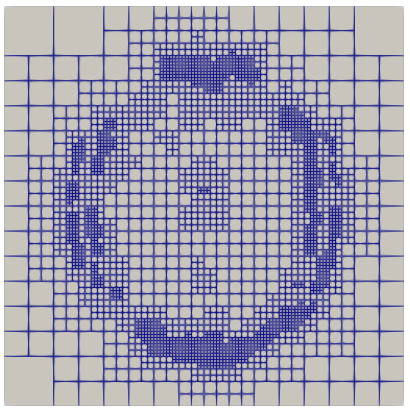}
    \end{subfigure}
    \begin{subfigure}{\wid}
        \centering
        \includegraphics[width=\textwidth]{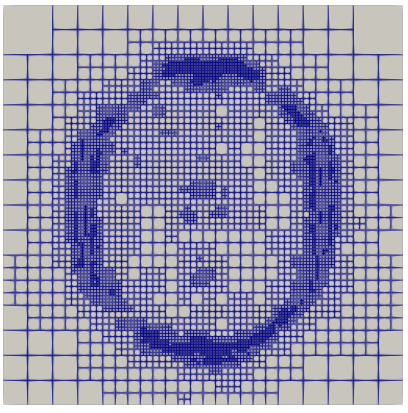}
    \end{subfigure}
    \begin{subfigure}{\wid}
        \centering
        \includegraphics[width=\textwidth]{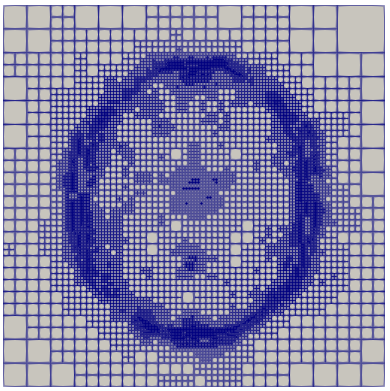}
    \end{subfigure}

    \begin{subfigure}{\wid}
        \centering
        \includegraphics[width=\textwidth]{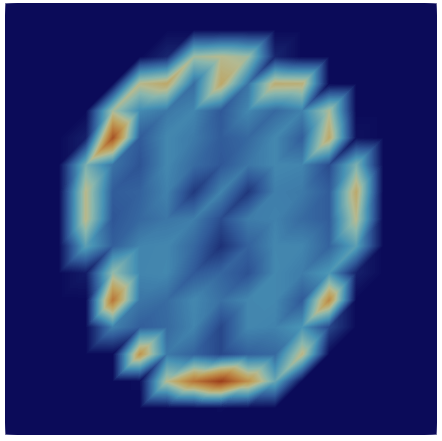}
    \end{subfigure}
    \begin{subfigure}{\wid}
        \centering
        \includegraphics[width=\textwidth]{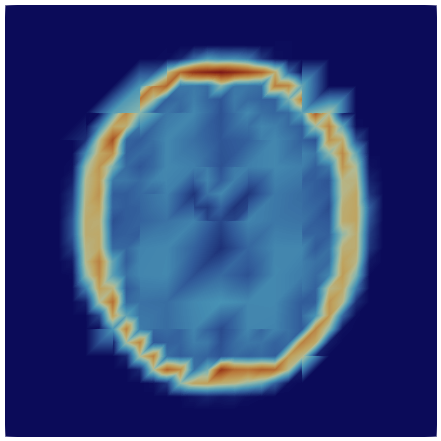}
    \end{subfigure}
    \begin{subfigure}{\wid}
        \centering
        \includegraphics[width=\textwidth]{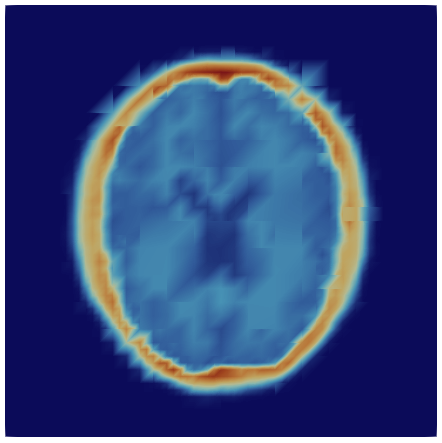}
    \end{subfigure}
    \begin{subfigure}{\wid}
        \centering
        \includegraphics[width=\textwidth]{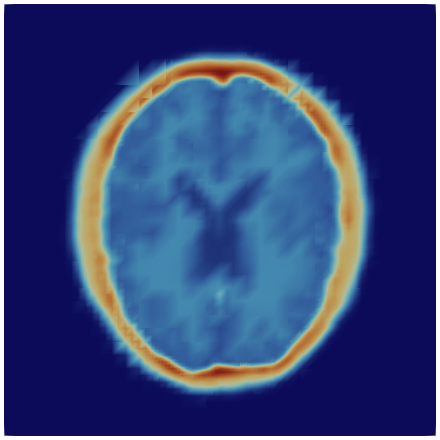}
    \end{subfigure}
    \begin{subfigure}{\wid}
        \centering
        \includegraphics[width=\textwidth]{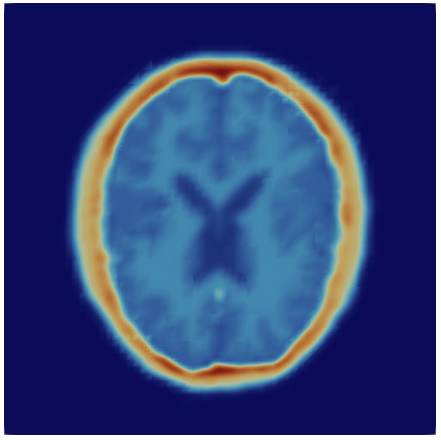}
    \end{subfigure}
    \begin{subfigure}{\wid}
        \centering
        \includegraphics[width=\textwidth]{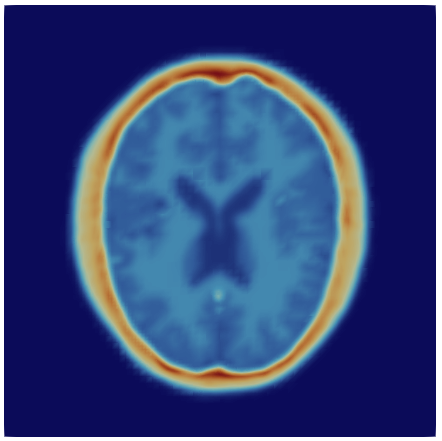}
    \end{subfigure}
    \caption{Brain \ac{amr} test: Evolution of \ac{amr} solution through all adaptive steps from left to right. In the top row we show the evolution of the adapted grid, and on the bottom row we show the evolution of the solution.}
    \label{fig:brain-amr-solution}
\end{figure}

In Figure~\ref{fig:brain-amr-solution} we show the evolution of the meshes obtained from our algorithm, and compare the solution obtained with a reference mesh of one-element-per-pixel in Figure~\ref{fig:brain-amr-comparison}. The solutions are indistinguishable to the eye, so we provide further comparison information in Table~\ref{table:brain-amr}, where we show the final number of \acp{dof}, similarity, total iterations, and required elapsed time. While the AMR solution requires a larger total number of iterations, most of these are performed on coarser meshes where iterations are much faster. This results in the overall solution time being accelerated by a factor of 13.62. This faster solution yielded a mildly lower similarity measure, and this was additionally achieved with less \acp{dof}. We stress that this test does not converge with the parameter values at hand if \ac{aa} is not used.

\begin{figure}[htbp]
    \newcommand{\wid}{0.3\textwidth}
    \centering
    \begin{subfigure}{\wid}
        \centering
        \includegraphics[width=\textwidth]{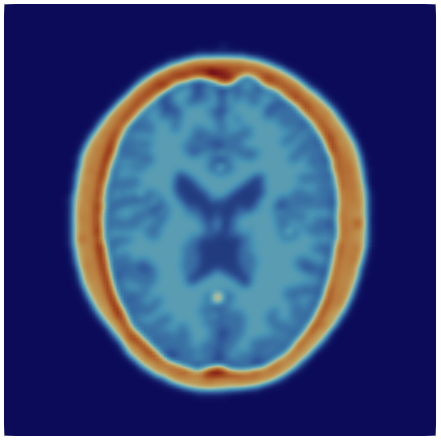}
        \caption{$R$}
    \end{subfigure}
    \begin{subfigure}{\wid}
        \centering
        \includegraphics[width=\textwidth]{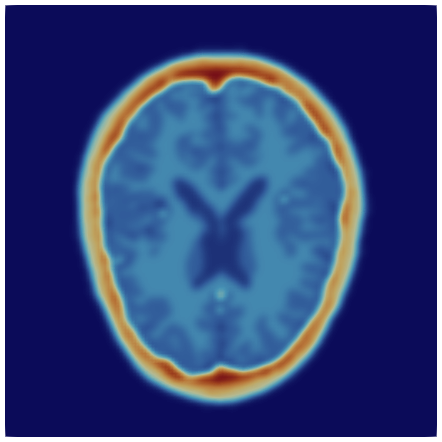}
        \caption{$T$}
    \end{subfigure}

    \begin{subfigure}{\wid}
        \centering
        \includegraphics[width=\textwidth]{brain-adaptive-5-sol.png}
        \caption{$T\circ(I + \vec u)$ \ac{amr}}
    \end{subfigure}
    \begin{subfigure}{\wid}
        \centering
        \includegraphics[width=\textwidth]{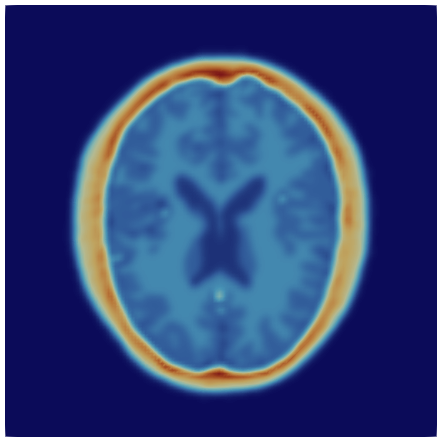}
        \caption{$T\circ(I + \vec u)$ classic}
    \end{subfigure}
    \caption{Brain \ac{amr} test: Comparison of \ac{amr} and classic solutions (bottom left and bottom right respectively) in comparison to original images (top row).}
    \label{fig:brain-amr-comparison}
\end{figure}

\begin{table}[htbp]
    \centering
    \begin{tabular}{r | r r r r}
        \toprule Strategy & \acp{dof} & Similarity & Iterations         & elapsed time (seconds) \\ \midrule
            Classic & 33,803 & 0.0263      & 927            & 1,649.94 \\
            Adaptive & 24,843 & 0.0260     &  \textcolor{gray}{{(\footnotesize 616, 321, 76, 43, 27, 24)}} 1,107 & 121.14
 \\ \bottomrule
    \end{tabular}
    \caption{Brain \ac{amr} test: Solution metrics, given by (a) total Degrees of Freedom (DoFs), (b) final similarity, (c) iterations, where we display the iterations incurred by the adaptive solver in each level, and (d) the overall elapsed time.}
    \label{table:brain-amr}
\end{table}

\subsection{Adaptivity performance on OC images}

The setup of the experiments for the OC images in this section is almost equivalent to that of the brain images, with the following differences.  First, we considered a tolerance of $10^{-2}$ for the residual in the reference (non-adapted) case, and a tolerance of $10^{-1}$ for all the adaptive cases. Second, we used $\alpha=10^8$ to obtain a more accurate solution than the one in Table~\ref{table:aa-oc-small} and \ref{table:aa-oc-large},  and set $\Delta t=10^{-6}$, which yields convergence (without \ac{aa}). Third, we did not use \ac{aa}, as it did not yield to a convergent fixed-point iteration scheme for the values of $\alpha$ and $\Delta t$ at hand.

In Figure~\ref{fig:oc-amr-solution} we show the evolution of the meshes obtained from our algorithm, and compare the solution obtained with a reference mesh of four-pixels-per-element in Figure~\ref{fig:oc-amr-comparison}. The solutions are almost indistinguishable to the eye, so we provide further comparison information in Table~\ref{table:oc-amr} (see description of Table~\ref{table:brain-amr} above for a description of the different fields in the table). We highlight that the overall solution time was accelerated by a factor of 2.09. This faster solution yielded a lower similarity measure, and this was achieved with less \acp{dof}.

\begin{figure}[htbp]
    \newcommand{\wid}{0.16\textwidth}
    \centering
    \begin{subfigure}{\wid}
        \centering
        \includegraphics[width=\textwidth]{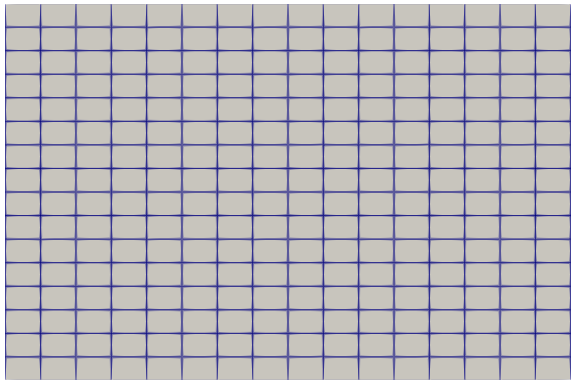}
    \end{subfigure}
    \begin{subfigure}{\wid}
        \centering
        \includegraphics[width=\textwidth]{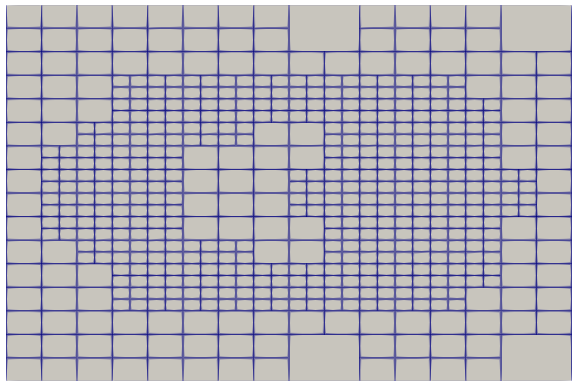}
    \end{subfigure}
    \begin{subfigure}{\wid}
        \centering
        \includegraphics[width=\textwidth]{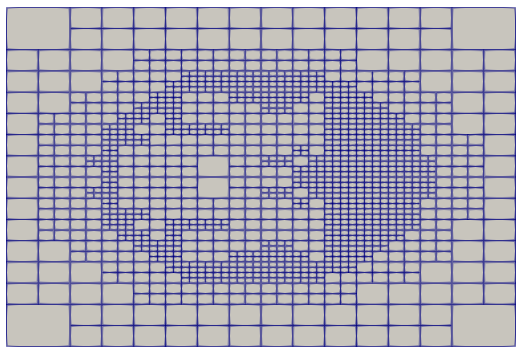}
    \end{subfigure}
    \begin{subfigure}{\wid}
        \centering
        \includegraphics[width=\textwidth]{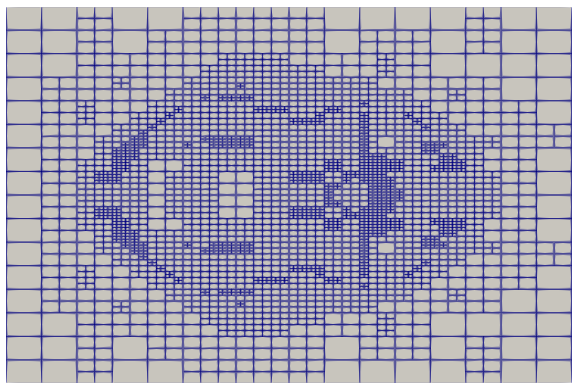}
    \end{subfigure}
    \begin{subfigure}{\wid}
        \centering
        \includegraphics[width=\textwidth]{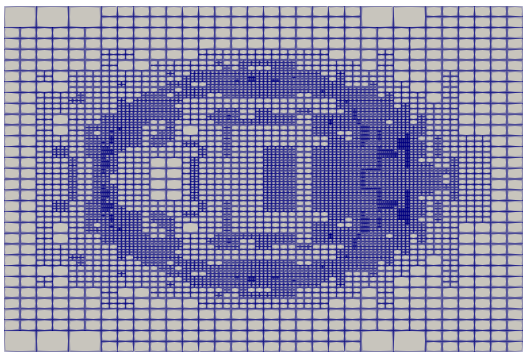}
    \end{subfigure}
    \begin{subfigure}{\wid}
        \centering
        \includegraphics[width=\textwidth]{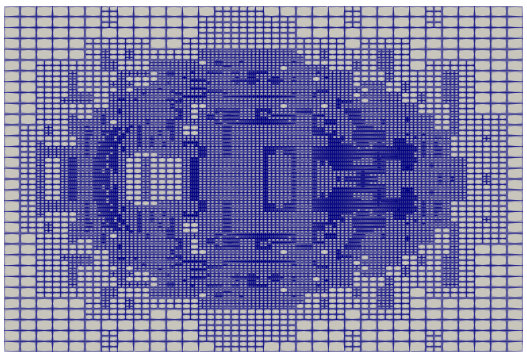}
    \end{subfigure}

    \begin{subfigure}{\wid}
        \centering
        \includegraphics[width=\textwidth]{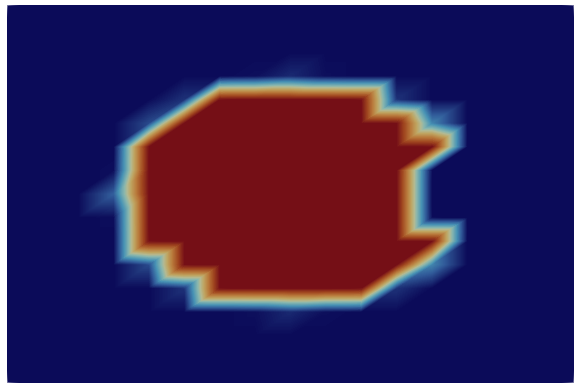}
    \end{subfigure}
    \begin{subfigure}{\wid}
        \centering
        \includegraphics[width=\textwidth]{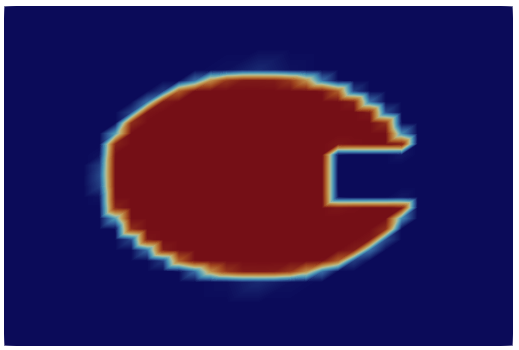}
    \end{subfigure}
    \begin{subfigure}{\wid}
        \centering
        \includegraphics[width=\textwidth]{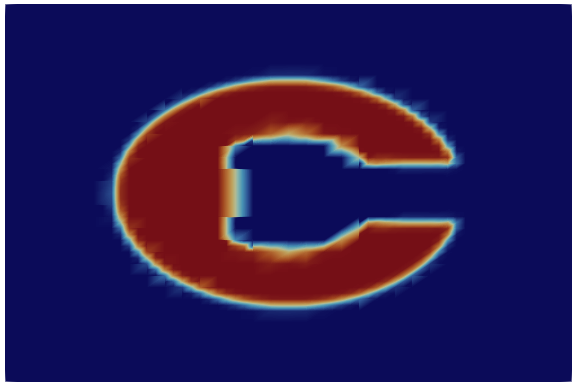}
    \end{subfigure}
    \begin{subfigure}{\wid}
        \centering
        \includegraphics[width=\textwidth]{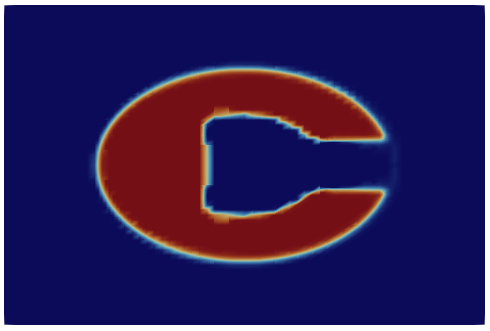}
    \end{subfigure}
    \begin{subfigure}{\wid}
        \centering
        \includegraphics[width=\textwidth]{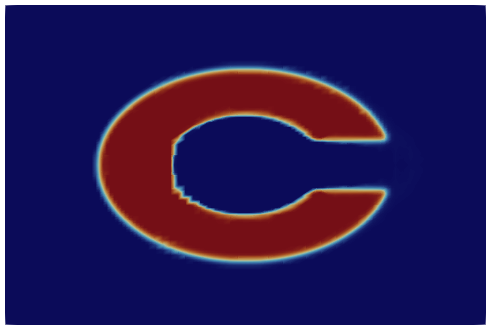}
    \end{subfigure}
    \begin{subfigure}{\wid}
        \centering
        \includegraphics[width=\textwidth]{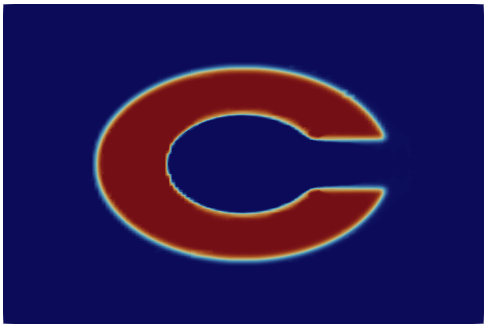}
    \end{subfigure}
    \caption{OC \ac{amr} test: Evolution of \ac{amr} solution through all adaptive steps from left to right. In the top row we show the evolution of the adapted grid, and on the bottom row we show the evolution of the solution.}
    \label{fig:oc-amr-solution}
\end{figure}

\begin{figure}[htbp]
    \newcommand{\wid}{0.3\textwidth}
    \centering
    \begin{subfigure}{\wid}
        \centering
        \includegraphics[width=\textwidth]{OC-R.png}
        \caption{$R$}
    \end{subfigure}
    \begin{subfigure}{\wid}
        \centering
        \includegraphics[width=\textwidth]{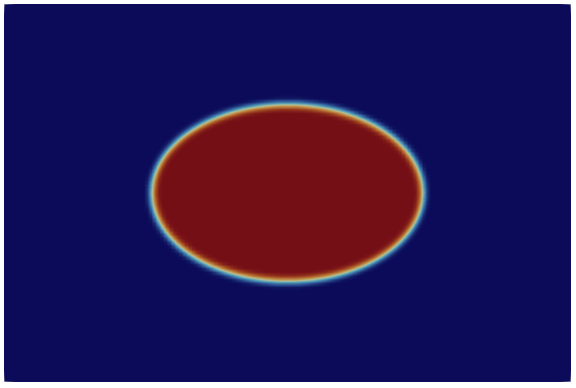}
        \caption{$T$}
    \end{subfigure}

    \begin{subfigure}{\wid}
        \centering
        \includegraphics[width=\textwidth]{OC-adaptive-5-sol.png}
        \caption{$T\circ(I + \vec u)$ \ac{amr}}
    \end{subfigure}
    \begin{subfigure}{\wid}
        \centering
        \includegraphics[width=\textwidth]{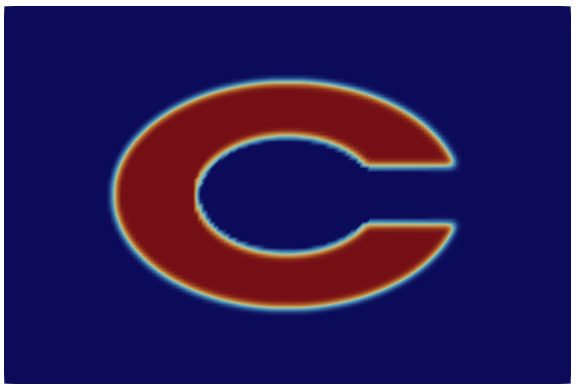}
        \caption{$T\circ(I + \vec u)$ classic}
    \end{subfigure}
    \caption{OC \ac{amr} test: Comparison of \ac{amr} and classic solutions (bottom left and bottom right respectively) in comparison to original images (top row).}
    \label{fig:oc-amr-comparison}
\end{figure}

\begin{table}[htbp]
\setlength{\tabcolsep}{4.pt}
    \centering
    \begin{tabular}{r | r r r r}
        \toprule Strategy & \acp{dof}   & Similarity & Iterations         & elapsed time (seconds) \\ \midrule
                Classic   & 30,505 & 0.0196     & 1,478               & 1,003.88 \\
                Adaptive   & 23,953 & 0.0156     & \textcolor{gray}{{(\footnotesize 21, 124, 549, 217, 662, 295)}} 1,868  & 478.57
 \\ \bottomrule
    \end{tabular}
    \caption{OC \ac{amr} test: Solution metrics, given by (a) total Degrees of Freedom (DoFs), (b) final similarity, (c) iterations, where we display the iterations incurred by the adaptive solver in each level in gray, and (d) the overall elapsed time.}
    \label{table:oc-amr}
\end{table}

\subsection{Quadrature sensitivity}\label{appendix:integration}

Image functions are highly nonlinear, which makes the computation of integrals that depend on them highly prone to numerical integration (quadrature) errors. It is difficult to obtain analytic results that can give sharp estimates for this, so instead we tried the following approach: depending of the pixels-per-element that one can have, we compute the quadrature order (i.e. maximum polynomial degree that is integrated exactly \cite{ern2004theory}) that provides minimal variations of the residual. Naturally, one would expect that the larger the number of pixels-per-element, the larger the required quadrature order to commit approximately the same amount of numerical integration error. To test this, we considered as the ground truth the residual vector 
    $$ [W(q)]_i = \int_\Omega (T(\bx) - R(\bx))\nabla T(\bx) \cdot \bv_i\, {\dx}_q,$$
where $\bv_i$ is a basis function of the considered \ac{fe} space, and we denote with ${\dx}_q$ the numerical integration performed with a quadrature of order $q$. With it, the error will be denoted with
    $$ e(q) = \frac{| W(q) - W(q_\text{truth})|}{|W(q_\text{truth})|}. $$

\begin{figure}[t!]
    \centering
    \newcommand{\wid}{0.495\textwidth}
    \begin{subfigure}{\wid}
        \centering
        \includegraphics[width=\textwidth]{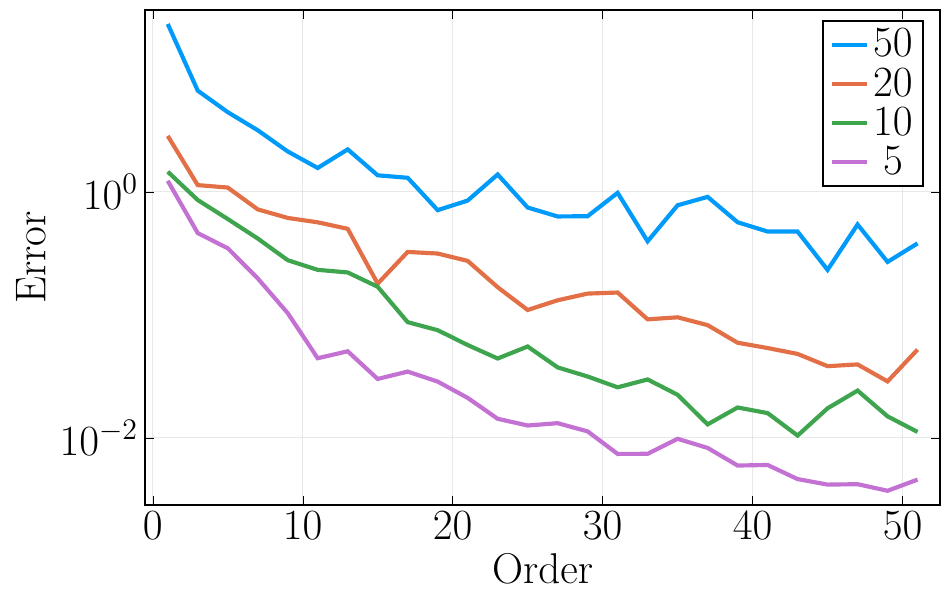}
        \caption{$\sigma=0$}
    \end{subfigure}
    \begin{subfigure}{\wid}
        \centering
        \includegraphics[width=\textwidth]{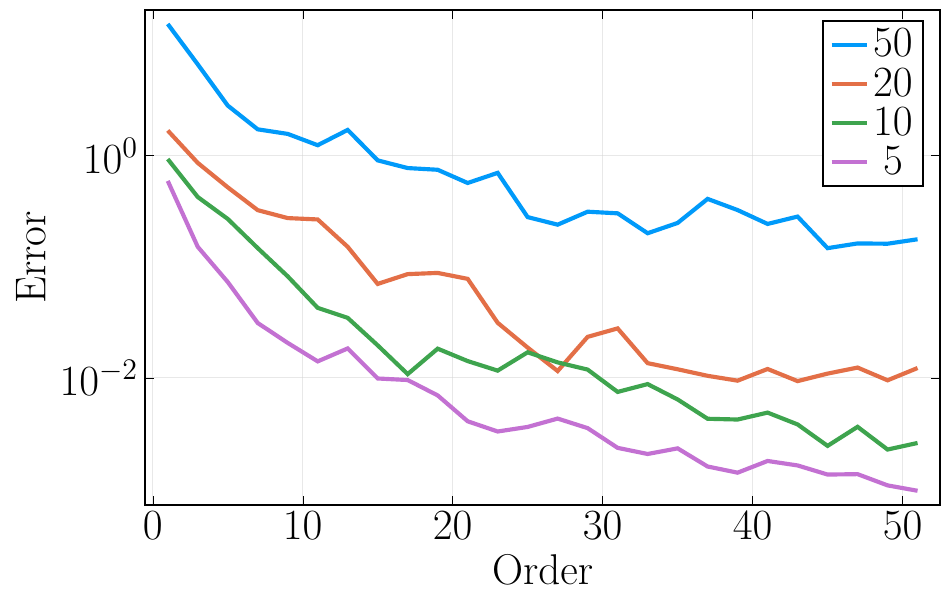}
        \caption{$\sigma=1$}
    \end{subfigure}

    \begin{subfigure}{\wid}
        \centering
        \includegraphics[width=\textwidth]{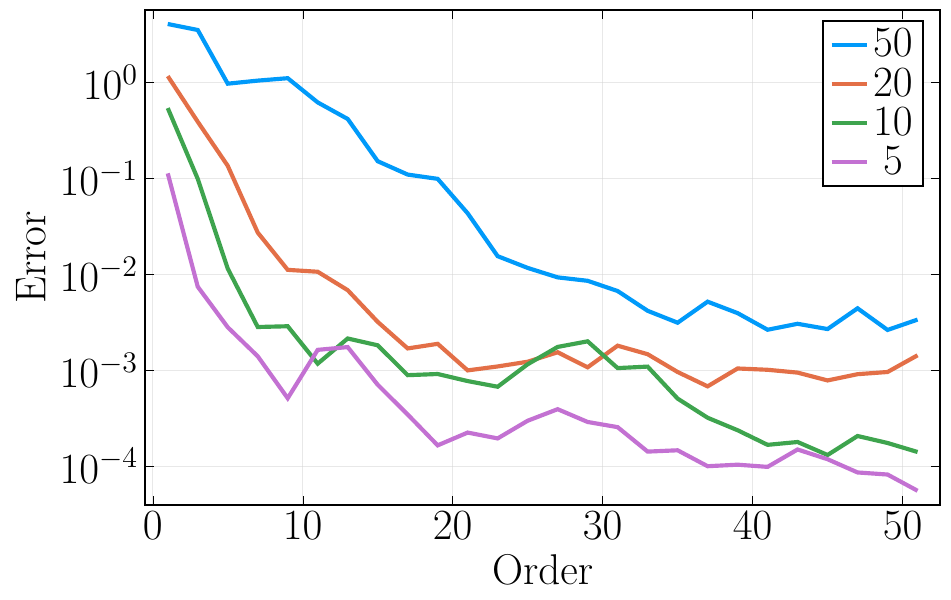}
        \caption{$\sigma=5$}
    \end{subfigure}
    \begin{subfigure}{\wid}
        \centering
        \includegraphics[width=\textwidth]{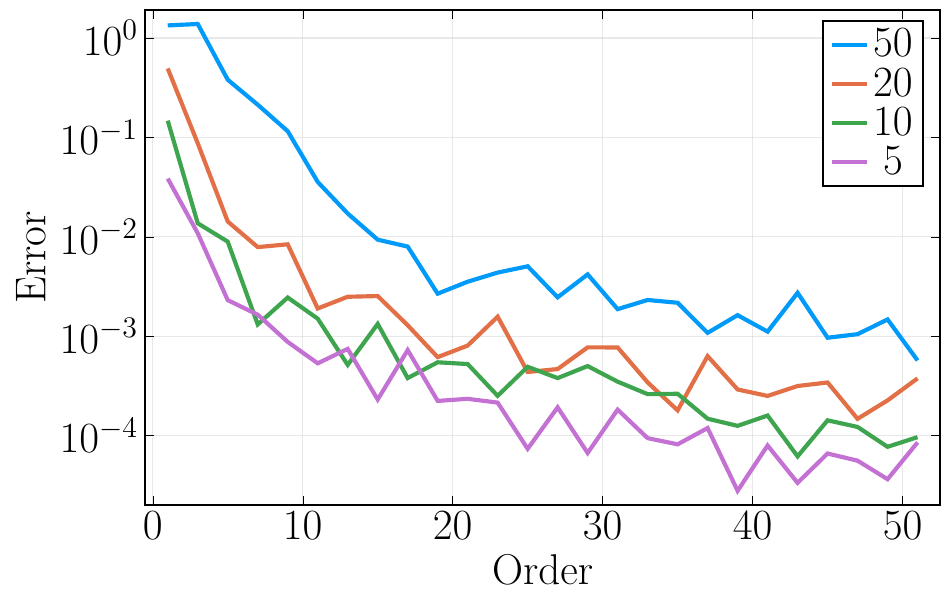}
        \caption{$\sigma=10$}
    \end{subfigure}
    \caption{Integration error for $5\times 5$, $10\times 10$, $20\times 20$, and $50\times 50$ pixels per element considering odd quadrature orders ranging from 1 to 51. From top to bottom, and left to right, we consider this setting using a Gaussian smoothing of 0, 1, 5, and 10 with the brain images. 
    }
    \label{fig:integration-errors}
\end{figure}

    We consider the ground truth integration value to be $q_\text{truth}=201$. In this setting, we computed the error for various numbers of pixels-per-element, and display the errors for different Gaussian smoothing levels and the brain images in Figure~\ref{fig:integration-errors}. 
    We use standard tensor-product Gauss-Legendre quadratures for the experiment.    
    Integration error is thus significant, and it becomes much more relevant when no smoothing is used. Considering that with \ac{amr} we have coarser elements where there are no complex feature of the images, we have found that a good compromise is considering $q=6$. Still, this will be addressed more effectively in future work.

\section{Conclusions and future perspectives}\label{section:conclusions}
In this work, we have extended the \emph{a-posteriori} error analysis of \ac{dir} with linear elastic regularisation to the case of Robin boundary conditions. We have formulated an efficient strategy for leveraging octree-based adaptivity that supports both coarsening and refinement, and have tested our strategy on realistic brain images and on the challenging OC benchmark. The \ac{dir} problem is highly nonlinear, and so we solve it with an IMEX formulation of a proximal-point algorithm, and further proposed to accelerate this algorithm with \ac{aa}. We have observed that acceleration can be very convenient computationally, as it significantly reduced the elapsed time of the brain test. Still, the methodology is highly problem dependent, as it was ineffective for the OC test when using larger values of $\alpha$. Putting everything together, we were able to significantly reduce the computational time required to solve \ac{dir}. The proposed \ac{amr} approach is able to provide better resolution on difficult domain regions and instead relax it where not required, which results in better solutions obtained with fewer \acp{dof} and in less time.

We will focus our future work on developing {black-box} nonlinear solvers for the \ac{dir} problem, so that we do not need to compute a pseudo-time step for convergence. We will provide a deeper analysis of the interplay between mesh size and solver performance, as the sensitivity of proximal-point algorithms with respect to mesh-size remains understudied, even more so in \ac{dir}.  Then, we will leverage such robust solvers to develop an automated computation of the similarity parameter $\alpha$, and extend our software to the 3D case. Our long term goal is that of developing an open source software that provides a robust and efficient solution of \ac{dir} without the requirement of intense parameter tuning.

\bibliographystyle{siam}
\bibliography{refs}

\end{document}